\documentclass[a4paper, reqno, 11pt]{amsart}
\makeatletter

\@addtoreset{equation}{section}
\makeatother
\usepackage{color}
\usepackage{amsmath}
\usepackage{amssymb}
\usepackage{amsthm}
\usepackage{graphicx}
\newtheorem{Thm}{Theorem}[section]

\newtheorem{Lem}[Thm]{Lemma}
\newtheorem{Prop}[Thm]{Proposition}

\theoremstyle{definition}

\newtheorem{Def}[Thm]{Definition}
\newtheorem{Rem}[Thm]{Remark}
\newtheorem{Exa}[Thm]{Example}

\newcommand{\Z}{\mathbb{Z}} 

\newcommand{\R}{\mathbb{R}}

\begin{document}

\title[Range on recurrent graph]{Some results for range of random walk on graph with spectral dimension two}
\author{Kazuki Okamura}
\address{School of General Education, Shinshu University, 3-1-1, Asahi, Matsumoto, Nagano, 390-8621, JAPAN.} 
\email{kazukio@shinshu-u.ac.jp} 

\begin{abstract} 
We consider the range of the simple random walk on graphs with spectral dimension two.  
We give a form of strong law of large numbers under a certain uniform condition, which is satisfied by not only the square integer lattice but also a  class of fractal graphs. 
Our results imply the strong law of large numbers on the square integer lattice established by Dvoretzky and Erd\"os (1951).     
Our proof does not depend on spacial homogeneity of space and gives a new proof of  the strong law of large numbers on the lattice.  

We also show that the behavior of appropriately scaled expectations of the range is stable with respect to every ``finite modification" of the two-dimensional integer lattice, and furthermore we construct a recurrent graph such that the uniform condition holds  but the scaled expectations fluctuate. 
As an application, we establish a form of law of the iterated logarithms for lamplighter random walks in the case that the spectral dimension of the underlying graph is two. 
\end{abstract}

\maketitle

\setcounter{tocdepth}{1}
\tableofcontents

\section{Introduction and main results}

Let $R_n$ be  the number of points which a random walk on an infinite connected graph visits up to time $n$. 
It is called the range of random walk. 
There are many results for ranges on the integer lattices $\Z^d$ (\cite{DE, JO, JP70, JP71, JP72-1, JP72-2, JP74, S, DV, L, LR, H98, H01, HK01, HK02, BK, BR, BCR, LV}).   
On the other hand, there are relatively few results for ranges on more general (deterministic or random) infinite connected simple graphs (\cite{BIK, G, KM, KoN, O14, R}).  
Difficulties for graphs, especially for {\it non} vertex-transitive graphs, arise from the lack of spacial homogeneity of graphs. 

A recurrent (resp. transient) graph is a graph on which the simple random walk on it is recurrent (resp. transient).    
Let $p_n (x, y)$ be the heat kernel of the simple random walk on a graph $G$.   
Define the spectral dimension of $G$ by 
\[ d_s (G) := -2\lim_{n \to \infty} \frac{\log p_{2n}(x,x)}{\log n}, \]  
if the limit exists.  
$G$ is transient if $d_s (G) > 2$ and $G$ is recurrent if $d_s (G) < 2$.   
This classification affects the growth rate of $\{R_n\}_n$.  
By \cite{DE}, the growth rates  of $\{R_n\}_n$ on $\Z^d, d = 1,2$, are strictly smaller than those on $\Z^d, d \ge 3$.

Ranges on some strongly recurrent\footnote{To the best of our knowledge, there is no unified definition of strongly recurrence. Notions of (very) strongly recurrent graphs are introduced in \cite{B04-1, T}. Their definitions are different from each other, and clearly different from the definition of null recurrence (i.e. the expectation for the first return time to the starting point is infinite. See \cite{W} for details.) for Markov chains.} graphs are considered by \cite{BJKS, KM, KoN, HHH} with motivations for verifying the Alexander-Orbach conjecture \cite{AO}. 
The random graphs appearing in their papers satisfy almost surely that $d_s (G) < 2$, and furthermore it holds that 
\[ \lim_{n \to \infty} \frac{\log R_n}{\log n} = \frac{d_s (G)}{2}, \textup{ a.s}.\]  
for the simple random walk on $G$. 

In \cite{O14}, the range on transient graphs satisfying a certain uniform condition are considered.  
\cite{O14} and \cite[Proposition 5.5]{KuN} extend the law of large numbers on $\Z^d, d \ge 3$,  to a certain class of graphs including not only $\Z^d$ but also some non-vertex transitive graphs such as the $d$-regular tree, and the infinite $d$-dimensional standard Sierpinski carpet graph, $d \ge 3$. 
As an immediate corollary of the law of large numbers, it holds that 
\[ \lim_{n \to \infty} \frac{\log R_n}{\log n} = 1, \textup{ a.s}.\]  
for those graphs. 

In this sense, the case that $d_s (G) = 2$ is a threshold case.  
In this paper, we consider the range of random walk on graphs with spectral dimension two. 
We show that if a recurrent graph satisfies a certain uniform condition, then, a certain strong law of large numbers for the range holds. 
The uniform condition is satisfied by a certain class of {\it non-}vertex-transitive graphs including isoradial graphs by \cite{K02} and fractal graphs by \cite{B04-1}, 
We remark that the uniform condition in \cite{O14} mentioned in the above paragraph is not suitable for recurrent graphs. 
Although we mainly focus on graphs which are {\it not} vertex-transitive, our results imply the strong law of large numbers on $\Z^2$ established by \cite{DE}. 

We show that the behavior of appropriately scaled expectations is stable with respect to every ``finite modification" of $\Z^2$.  
We also construct a recurrent graph satisfies the uniform condition but appropriately scaled expectations fluctuate.   
We finally apply our results to a partial law of the iterated logarithms for a random walk on a lamplighter graph when the spectral dimension of the underlying graph is two, which is remained to be open in \cite{KuN}. 

In the case that $G = \Z^2$, properties of the range such as law of large numbers, central limit theorem, law of the iterated logarithm, and large and moderate deviations are investigated by \cite{DE, JP70, JP72-2,  L, LR, HK01, BK, BR, BCR, LV}, 
however, the techniques used in those references heavily depend on spacial homogeneity of $\Z^2$. 
The techniques used for the range on strongly recurrent graphs in  \cite{BJKS, KM, KoN, HHH} do not work well enough for our purposes. 
Our proof of the {\it upper} bound for $R_n$ utilizes a certain large deviation technique similar to \cite{O14} which deals with the transient case. 

As in the case that $G = \Z^2$ in \cite{DE, JP70}, the most technically involved part is the proof of the {\it lower} bound for $R_n$. 
Our technical novelty is in the proof of the lower bound. 
We first show two preliminary results of weak forms in Theorem \ref{WLLN-UC} (ii) and (iii) below. 
Second, by assuming more regularities for graphs,  
we establish a result for self-intersections of a single random walk in Proposition \ref{self-intersect-main} below by utilizing some estimates for the simple random walk on graphs satisfying the Ahlfors regularity and (sub-)Gaussian heat kernel estimates.  
Contrary to \cite{L} dealing with the case that $G = \Z^2$,  self-intersections of the simple random walk have to be distinguished with intersections of two independent random walks starting from a common point. 
In Section 4, by combining Theorem \ref{WLLN-UC} (ii) and (iii) and Proposition \ref{self-intersect-main}, 
we perform a certain form of mathematical induction with respect to magnitudes of bounds and obtain the desired lower bound, which is stated  in Theorem \ref{Lower} below. 
Our strategy of the proof is different from \cite{DE, JP70, JP72-2, L, LR}. 
Specifically, we do not use any variance estimates as in \cite{JP72-2, L, LR} or deal with intersections of four different pieces of random walk traces as in \cite{DE, JP70}. 
In these senses, our proof of the strong law of large numbers is somewhat simpler than \cite{DE, JP70, JP72-2, L, LR}. 

\subsection{Framework and main results}  

Let $G = (V(G), E(G))$ be an infinite connected simple graph with bounded degrees. 
Denote the degree of the vertex $z$ by $\deg(z)$. 
Let $(S_n, P^x)$ be the simple random walk on $G$, that is, it is a Markov chain on $G$ whose transition probabilities are given by 
\[ P^x (S_{n+1} = y | S_n = z) = \frac{1}{\deg(z)}, \ P^x (S_0 = x) = 1, \] 
hold for every $n \ge 0$, $x, y, z \in V(G)$ such that $\{y, z\} \in E(G)$. 
Throughout this paper, we consider the simple random walk unless otherwise stated. 
Let $E^x$ be the expectation of $P^x$. 
Let the transition probability of $(S_n, P^x)$ be 
\[ p_n (x, y) := P^x (S_n = y), \ x, y \in G, n \ge 0.\]  
Let 
\[ T_A := \inf\{ k \ge 1 : S_k \in A \}, \ A \subset G. \]
For $A \subset G$, we denote the complement of $A$ with respect to $G$ by $A^c$.   
Let $d$ be the graph distance of $G$ and $B(x,r)$ be the closed ball with center $x$ and radius $r$, specifically, 
\[ B(x,r) = \{y \in V(G) : d(x,y) \le r\}.\] 
For $r \in \mathbb{R}$, 
let $\left\lfloor r \right\rfloor$ be the unique integer such that $\left\lfloor r \right\rfloor \le r < \left\lfloor r \right\rfloor +1$. 

In this paper, we use the term ``(uniformly) weakly recurrent" in the following sense. 
\begin{Def}
(i) We say that $G$ is {\it weakly recurrent} if for each $x \in G$, there exist positive constants $c_x$ and $C_x$ such that for every $n \ge 1$,  
\[ c_x \le n (p_{n} (x,x) + p_{n+1}(x,x)) \le C_x.  \]
(ii) We say that $G$ is {\it uniformly weakly recurrent} if there exist positive constants $c$ and $C$ such that for every $n \ge 1$ and \textit{every} $x \in G$,  
\begin{equation}\label{uc-ul}   
c \le n (p_{n} (x,x) + p_{n+1}(x,x)) \le C. 
\end{equation} 
\end{Def}
Condition (ii) above is the uniform condition for graphs mentioned in the introduction. 

\begin{Def}\label{infsup}
Let 
\[ r_{\inf} := \liminf_{n \to \infty} (\log n) \inf_{x \in G} P^x (T_x > n), \]
and 
\[ r_{\sup} := \limsup_{n \to \infty} (\log n) \sup_{x \in G} P^x (T_x > n). \]
\end{Def}
We prove in Lemma \ref{uc} that if  $G$ is uniformly weakly recurrent, then,  
\begin{equation}\label{posi-fi}
0 < r_{\textup{inf}} \le r_{\textup{sup}} < +\infty. 
\end{equation} 
For $G = \Z^2$, 
\[ r_{\textup{inf}} = r_{\textup{sup}} = \pi. \]

Let $R_n := |\{S_0, \dots, S_n\}|$, which is the main object of this paper. 

\begin{Thm}[Law of large numbers]\label{WLLN-UC}  
If $G$ is  uniformly weakly recurrent, then, we have the following:\\
(i) 
\[ \limsup_{n \to \infty} \frac{R_n}{n / \log n} \le r_{\sup}, \ \textup{ $P^x$-a.s.}  \]
(ii) 
For every $\epsilon > 0$, 
\begin{equation*}
\lim_{n \to \infty} P^x\left( \frac{R_n}{n / \log n} \le r_{\inf} - \epsilon \right) = 0. 
\end{equation*}
(iii)
\begin{equation*}
\liminf_{n \to \infty} \frac{R_n}{n / \log n} \ge \frac{r_{\inf}}{2}, \ \textup{ $P^x$-a.s.}  
\end{equation*}
(iv) 
\[ r_{\inf} \le \liminf_{n \to \infty} \frac{E^x [R_n]}{n / \log n} \le \limsup_{n \to \infty} \frac{E^x [R_n]}{n / \log n} \le r_{\sup}. \]
\end{Thm} 

These assertions hold also for every Markov chain satisfying \eqref{uc-ul}.  
Assertions (i) and (ii) correspond to \cite[Theorem 1.2]{O14} which deal with the transient case.  
The outline of our proof of (i) and (ii) is same as  \cite[Theorem 1.2]{O14}, however, in this recurrent case the decay rates appearing in our proof are slower than the transient case and  we need a little more complicated arguments.   
By this reason, we cannot adopt the method used in the proof of  \cite[Proposition 5.8]{KuN} for the proof of (iii). 

Assertion (ii) is more difficult to show than (i). 
Our proof of (i) need to assume only \eqref{posi-fi} and do not need to assume \eqref{uc-ul}. 
However, we need \eqref{uc-ul} for the proof of (ii). 

Assertion (iii)  is weaker than a better estimate in Theorem \ref{Lower}, under some additional assumptions for $G$. 
However it plays a crucial role in the proof of Theorem \ref{Lower}.  

If $G =\Z^2$, then, the proof of the strong law of large numbers  in \cite{DE, JP70} depends on variance estimates.  
We show Theorem \ref{WLLN-UC} by certain tail estimates and do not use any variance estimates. 

It is now natural to ask some properties of the limit infimum and supremum of $\{R_n / (n/\log n)\}_n$.   
For this, we impose some conditions for graphs stronger than uniformly weakly recurrence. 

Let $\alpha, \beta > 1$. 
We say that $G$ is {\it Ahlfors $\alpha$-regular} if there exist two positive constants $c_1$ and $c_2$ such that for every $n \ge 1$ and every $x \in V(G)$, 
\begin{equation}\label{Ahlfors} 
c_1 n^{\alpha} \le |B(x,n)| \le c_2 n^{\alpha}. 
\end{equation} 

We say that $G$ satisfies {\it sub-Gaussian heat kernel upper bounds with walk dimension $\beta$}, that is, 
there exist four positive constants $c_3, c_4, c_5, \textup{ and } c_6$ such that
\begin{equation}\label{sGHK-upper} 
p_{n}(x,y) \le \frac{c_3}{|B(x,n^{1/\beta})|} \exp\left(- c_4 \frac{d(x,y)^{\beta/(\beta-1)}}{n^{1/(\beta-1)}}\right),  
\end{equation}
holds for every $n \ge 1$ and every $x, y \in V(G)$. 
 
We say that $G$ satisfies {\it sub-Gaussian heat kernel lower bounds with walk dimension $\beta$}, that is, 
\begin{equation}\label{sGHK-lower} 
p_{n}(x,y) + p_{n+1}(x,y) \ge \frac{c_5}{|B(x,n^{1/\beta})|} \exp\left(- c_6  \frac{d(x,y)^{\beta/(\beta-1)}}{n^{1/(\beta-1)}} \right), 
\end{equation}
holds for every $n \ge 1$ and every $x, y \in V(G)$ satisfying that $n \ge d(x,y)$. 

\begin{Thm}\label{Lower}
Let $G$ be a uniformly weakly recurrent graph. 
Assume that for some $d > 1$,  $G$ is Ahlfors $d$-regular and satisfies the sub-Gaussian heat kernel upper bounds with walk dimension $d$.  
Then, for every $x \in V(G)$, 
\begin{equation*}
\liminf_{n \to \infty} \frac{R_n}{n / \log n} \ge r_{\inf}, \ \textup{ $P^x$-a.s.}  
\end{equation*}
\end{Thm}

For $G = \Z^2$, by Theorems \ref{WLLN-UC} and \ref{Lower}, we have the strong law of large numbers  by \cite{DE}.

We remark that if $G$ is uniformly weakly recurrent, then, we also have the off-diagonal heat kernel upper bound \eqref{sGHK-upper} with $\beta = 2$. 
This fact is originally from \cite{HSC}. 
See also Woess \cite[Theorem 14.12]{W} and Barlow \cite[Theorem 6.29]{B17}. 


We now consider the values of $\displaystyle\limsup_{n \to \infty} \frac{R_n}{n / \log n}$ and $\displaystyle\liminf_{n \to \infty} \frac{R_n}{n / \log n}$. 
\begin{Thm}\label{SLLN} 
Assume that $(S_n, P^x)$ is the simple random walk on $G$. 
Assume that for some $d > 1$, $G$ is Ahlfors $d$-regular and satisfies the sub-Gaussian heat kernel upper and lower bounds with walk dimension $d$. 
Then, there exist two non-random constants $c_{\inf}$ and $c_{\sup}$ such that for every $x \in V(G)$, 
\[ \liminf_{n \to \infty} \frac{R_n}{n / \log n} = c_{\inf}, \ \textup{ $P^x$-a.s.}  \]
\[ \limsup_{n \to \infty} \frac{R_n}{n / \log n} = c_{\sup}, \ \textup{ $P^x$-a.s.}  \] 
Furthermore,
\begin{equation}\label{c-inf-Fatou-new} 
r_{\inf} \le c_{\inf} \le \liminf_{n \to \infty} \frac{E^x[R_n]}{n / \log n}. 
\end{equation}
\begin{equation}\label{sandwich-new}  
r_{\inf} \le c_{\inf} \le c_{\sup} \le r_{\sup}. 
\end{equation} 
\end{Thm}

We emphasize that $c_{\inf}$ and $c_{\sup}$ are independent from the starting point $x$ of the random walk. 
It is difficult to calculate concrete forms of $c_{\inf}$ and $c_{\sup}$.  
Theorem \ref{WLLN-UC} does not exclude the possibility that the limit of $\dfrac{E^x [R_n]}{n / \log n} $ or $\dfrac{R_n}{n / \log n}$ as $n \to \infty$ exists. 
We now give an example of graph for which the limit does {\it not} exist. 

\begin{Def}\label{def-fm} 
We say that $G^{\prime}$ is a {\it finite modification} of $G$ 
if 
there exist two finite subsets $D$ on the set of vertices of $G$ and $D^{\prime}$ on the set of vertices of $G^{\prime}$ 
such that there is a graph isomorphism (see Diestel  \cite[Section 1.1]{Di} for the definition of this terminology) $\phi : V(G) \setminus D  \to V(G^{\prime}) \setminus D^{\prime}$.    
As graph structures for $V(G) \setminus D$ and $V(G^{\prime}) \setminus D^{\prime}$, 
we consider the induced subgraphs of $G$ and $G^{\prime}$, respectively.   
We denote them by $G \setminus D$ and $G^{\prime} \setminus D^{\prime}$, respectively. 
\end{Def}  

We remark that if $G^{\prime}$ is a {\it finite modification} of $G$, then $G^{\prime}$ is roughly isometric to $G$. 
See Woess \cite[Definition 3.7]{W} for the definition of being roughly isometric.
We remark that $G \setminus D$ and $G^{\prime} \setminus D^{\prime}$ can be non-connected. (For example, consider the case $G$ and $G^{\prime}$ are trees.) 
However, in this paper, we consider finite modifications of two specific graphs such that its vertices is $\Z^2$ and its edges contains all nearest-neighbor edges, so $G \setminus D$ and $G^{\prime} \setminus D^{\prime}$ are always connected.    

\begin{center}
\begin{picture}(150,150)
\linethickness{1pt}
\multiput(0,0)(10, 0){5}{\line(0,1){140}}
\multiput(0,0)(0,10){5}{\line(1,0){140}}
\multiput(100,0)(10, 0){5}{\line(0,1){140}}
\multiput(0,100)(0,10){5}{\line(1,0){140}}
\multiput(100,0)(-10, 0){6}{\line(0,1){40}}
\multiput(0,100)(0,-10){6}{\line(1,0){40}}
\multiput(100,100)(-10, 0){6}{\line(0,1){40}}
\multiput(100,100)(0,-10){6}{\line(1,0){40}}
\put(40,70){\line(1,0){60}}
\put(70,60){\line(0,1){40}}
\put(70,80){\line(1,0){30}}
\put(40,60){\line(1,-1){20}}
\put(70,40){\line(1,1){30}}
\put(80,50){\line(0,1){20}}
\put(80,80){\line(0,1){10}}
\put(50,50){\line(1,0){30}}
\put(60,80){\line(0,1){20}}
\end{picture} 
\end{center} 
\begin{center}
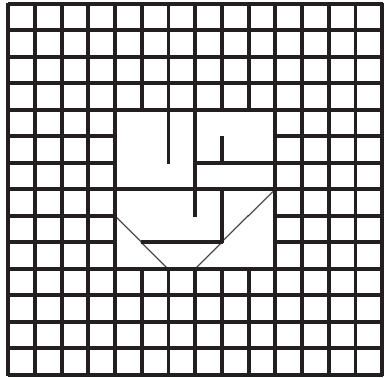
an illustration of a finite modification of $\Z^2$
\end{center}

\begin{Thm}\label{fm-2d}
If $G$ is a finite modification of $\Z^2$, then,  for the simple random walk on $G$,  
\[ \lim_{n \to \infty} \frac{\log n}{n} E^o [R_n] = \pi. \]
\end{Thm} 

We also show the following by using Theorem \ref{fm-2d}. 

\begin{Thm}\label{fluc-2d}
There exist a uniformly weakly recurrent graph $G$ and  a vertex $o$ of $G$ such that for the simple random walk on $G$, 
\begin{equation}\label{fluc-2d-p0}  
\liminf_{n \to \infty} \frac{\log n}{n} E^{o}[R_n]  <  \limsup_{n \to \infty} \frac{\log n}{n} E^{o}[R_n], 
\end{equation} 
and for some constants $c_{\inf}$ and $c_{\sup}$, 
\begin{equation}\label{SLLN-fluc}
0 < \liminf_{n \to \infty} \frac{\log n}{n} R_n = c_{\inf} < c_{\sup} = \limsup_{n \to \infty} \frac{\log n}{n} R_n  < +\infty, \ \textup{$P^o$-a.s.}
\end{equation}
\end{Thm} 

In the course of the proof of this theorem, we show that $G$ in the statement of Theorem \ref{fluc-2d} is Ahlfors $2$-regular and satisfies the sub-Gaussian heat kernel estimates with walk dimension $2$.  
Hence $c_{\inf}$ and $c_{\sup}$ are the same constants as in Theorem \ref{SLLN}.

This paper is organized as follows. 
In Section 2, we show Theorem \ref{WLLN-UC}. 
In Section 3, we give some results for intersection of random walks on graphs. 
In Section 4, we show Theorem \ref{Lower} by using the results obtained in Section 3. 
In Section 5, which is shorter than any other sections,  we show Theorem \ref{SLLN}. 
In Section 6, we show Theorems \ref{fm-2d} and \ref{fluc-2d}.
Finally in Section 7, we give examples and applications. 


\section{Proof of Theorem \ref{WLLN-UC}}

Hereafter, for two infinite sequences $(f(n))_n$ and $(g(n))_n$ satisfying that $g(n) \ne 0$ for every $n$, 
$f(n) = O(g(n))$ means that \[ \limsup_{n \to \infty} \left| \frac{f(n)}{g(n)} \right| < +\infty,  \]
and, 
$f(n) = o(g(n))$ means that \[ \lim_{n \to \infty} \left|\frac{f(n)}{g(n)}\right|  = 0.  \]

\begin{Lem}\label{uc}
(i) Let $x \in G$. 
If there exist positive constants $c_x$ and $C_x$ such that for every $n \ge 1$,  
\[ c_x \le n (p_{n}(x,x) + p_{n+1}(x,x)) \le C_x,  \]
then, 
\begin{equation}\label{relation}
\lim_{n \to \infty} P^x (T_x > n) \sum_{k = 1}^{n} p_{k}(x,x) = 1. 
\end{equation} 
(ii) If \eqref{uc-ul} holds, 
then, the  convergence \eqref{relation}  is uniform with respect to $x$, and furthermore, \eqref{posi-fi} holds. \\
(iii) If $G$ is vertex-transitive and furthermore 
$$\lim_{k \to +\infty} k p_{2k}(x,x) = \frac{1}{r}$$ 
for some $x \in G$ and $r > 0$, 
then, $r = r_{\inf} = r_{\sup}$. 
\end{Lem}

Assertion (iii) is applicable to the case that $G$ is $\mathbb{Z}^2$.  

\begin{proof} 
(i) Let $L^{(n)}_x$ be the last hitting time of $\{S_n\}_n$ to a point $x$ up to time $n$.  
Then by the Markov property,  
\[ P^x (L^{(n)}_x = k) = p_k (x,x) P^x (T_x > n -k), \ 0 \le k \le n.  \]
By using this and 
\[ \sum_{k = 0}^{n}  P^x (L^{(n)}_x = k) = 1,  \]
we have that 
\[ \sum_{k = 0}^{n} p_{k}(x,x)  P^x (T_x > n -k) = 1.  \]
Hence we have  that for {\it every} $n$, 
\begin{equation}\label{hit-hk-easy} 
P^x (T_x > n) \sum_{k = 0}^{n} p_{k}(x,x) \le 1, 
\end{equation}
and furthermore, 
\begin{equation}\label{hit-hk-difficult} 
P^x \left(T_x > \left\lfloor\frac{n}{\log n}\right\rfloor\right) \sum_{k =0}^{n - \left\lfloor n/\log n\right\rfloor -1} p_k (x,x) \ge 1 - \sum_{k = n - \left\lfloor n/\log n\right\rfloor}^{n} p_k (x,x). 
\end{equation}

By the upper bound in \eqref{uc-ul},  
\[ \lim_{n \to \infty} \sum_{k = n - \left\lfloor n/\log n\right\rfloor}^{n} p_k (x,x) = 0,  \]    
By this and \eqref{hit-hk-easy},  
we have  that 
\begin{equation}\label{pre-lower}  
\liminf_{n \to \infty} P^x \left(T_x >  \left\lfloor\frac{n}{\log n}\right\rfloor \right) \sum_{k = 0}^{n - \left\lfloor n/\log n\right\rfloor -1} p_{k}(x,x) \ge 1. 
\end{equation} 

By the lower estimate in \eqref{uc-ul}, 
\[ \liminf_{n \to \infty} \frac{1}{\log n} \sum_{k = 0}^{\left\lfloor n/\log n\right\rfloor -1} p_{k}(x,x) > 0. \]

By the upper estimate in \eqref{uc-ul}, 
\[ \limsup_{n \to \infty} \frac{1}{\log \log n} \sum_{k = \left\lfloor n/\log n\right\rfloor}^{n - \left\lfloor n/\log n\right\rfloor -1} p_{k}(x,x) < +\infty. \]
Hence, 
\begin{equation*}
\lim_{n \to \infty} \frac{ \sum_{k = 0}^{\left\lfloor n/\log n\right\rfloor -1} p_{k}(x,x)}{ \sum_{k = 0}^{n - \left\lfloor n/\log n\right\rfloor -1} p_{k}(x,x)} 
= \lim_{n \to \infty} \frac{1}{1 + \frac{\sum_{k = \left\lfloor n/\log n\right\rfloor}^{n - \left\lfloor n/\log n\right\rfloor -1} p_{k}(x,x)}{\sum_{k = 0}^{\left\lfloor n/\log n\right\rfloor -1} p_{k}(x,x)}} =  1.
\end{equation*}

By this and \eqref{pre-lower}, 
we have  that 
\[ \liminf_{n \to \infty} P^x \left(T_x > \left\lfloor\frac{n}{\log n}\right\rfloor\right) \sum_{k = 0}^{\left\lfloor n/\log n\right\rfloor - 1} p_{k}(x,x) \ge 1. \]
This and \eqref{hit-hk-easy} implies the assertion. 

(ii) By the uniformity assumption, there is a constant $C > 0$ such that  for every $n \ge 3$ and every $x \in V(G)$,  
\begin{equation}\label{hk-upper-uni} 
\sum_{k = n - \left\lfloor n/\log n \right\rfloor}^{n} p_k (x,x) \le  \frac{C}{\log n}, 
\end{equation} 
and,   
\begin{equation}\label{hk-upper-uni-2}  
P^x \left(T_x > \left\lfloor \frac{n}{\log n} \right\rfloor\right) \sum_{k = \left\lfloor n/\log n\right\rfloor}^{n - \left\lfloor n/\log n\right\rfloor-1} p_{k}(x,x) \le C \frac{\log\log n}{\log n - \log\log n}. 
\end{equation} 

Then, by \eqref{hk-upper-uni} and \eqref{hit-hk-difficult},  
\[ P^x \left(T_x > \left\lfloor\frac{n}{\log n}\right\rfloor\right) \sum_{k =0}^{\left\lfloor n/\log n\right\rfloor -1} + \sum_{k = \left\lfloor n/\log n\right\rfloor}^{n - \left\lfloor n/\log n\right\rfloor-1} p_k (x,x) \ge 1 -   \frac{C}{\log n}. \]
By this and \eqref{hk-upper-uni-2},  
\[ P^x \left(T_x > \left\lfloor\frac{n}{\log n}\right\rfloor\right) \sum_{k =0}^{\left\lfloor n/\log n\right\rfloor-1} p_k (x,x) \ge 1 -   \frac{C}{\log n} - C \frac{\log\log n}{\log n - \log\log n}. \]
Thus we have 
\begin{align}\label{uniform-error} 
\left| P^x \left(T_x > \left\lfloor\frac{n}{\log n}\right\rfloor\right)  \sum_{k =0}^{\left\lfloor n/\log n\right\rfloor-1} p_k (x,x)  -1 \right| &\le O\left(\frac{\log \log n}{\log n}\right) \notag\\
&=  O\left(\frac{\log \log \left\lfloor n/\log n\right\rfloor}{\log \left\lfloor n/\log n\right\rfloor}\right). 
\end{align} 

(iii) By the Abelian theorem, we have that 
\[ \lim_{n \to \infty} \frac{1}{\log n} \sum_{k = 1}^{n} p_{k}(x,x) = \frac{1}{r}. \]
The assertion follows from this and (i). 
\end{proof}

The following proof is similar to the proof of \cite[Theorem 1.2]{O14}. 
Informally speaking, the reason we can use the techniques of the proof of \cite[Theorem 1.2]{O14}  is that weakly recurrent graphs are ``close" to transient graphs. 
Our methods are not  applicable to strongly recurrent graphs.    

\begin{proof}[Proof of Theorem \ref{WLLN-UC} (i)]  
Fix $x \in G$ and $\epsilon > 0$. 
Let 
\[ Y_{i, j} := 1_{\{S_i \ne S_{k}, \ i < k \le i+j\}}. \]

Let $M = M_n$ be an integer specified later such that $M_n = o (n/\log n)$.   
By the last exit decomposition,   
\[ R_{n} = 1 + \sum_{i=0}^{n-2} Y_{i, n-1-i}
\le M_n + \sum_{i=0}^{n-1-M_n} Y_{i, n-1-i}
\le  M_n + \sum_{i=0}^{n-1-M_n} Y_{i, M_n}. \]
Hence, for large $n$, by using $M_n < n\epsilon/(2\log n)$, 
\begin{align*} 
P^x \left(R_{n} \ge (n/\log n) (r_{\sup} + \epsilon)\right) 
&\le P^x \left(\sum_{i=0}^{n-1-M_n} Y_{i, M_n} > \frac{n}{\log n} \left(r_{\sup}+  \frac{\epsilon}{2}\right)\right) \\
&= P^x \left(\sum_{a=0}^{M_n} \sum_{i \equiv a \!\!\!\!\mod(M_n +1)}Y_{i, M_n} > \frac{n}{\log n} \left(r_{\sup} +  \frac{\epsilon}{2}\right) \right)\\
&\le \sum_{a=0}^{M_n} P^x  \left(\sum_{i \equiv a \!\!\!\!\mod(M_n +1)}Y_{i, M_n} > \frac{n}{(M_n +1) \log n}\left(r_{\sup} + \frac{\epsilon}{2}\right)\right). 
\end{align*}

For every $t > 0$, by the exponential Chebyshev inequality, we have that for each integer $a$, 
\[ P^x \left(\sum_{i \equiv a \!\!\!\!\mod(M_n +1)}Y_{i, M_n} >  \frac{n/ \log n}{M_n +1}\left(r_{\sup} + \frac{\epsilon}{2}\right)\right) \]
\begin{equation}\label{sharp} 
\le \exp\left(-t  \frac{n}{(M_n +1)\log n}\left(r_{\sup} + \frac{\epsilon}{2}\right) \right)
E^x \left[\exp\left( t \sum_{i \equiv a \!\!\!\!\mod(M_n +1)}Y_{i, M_n} \right)\right].  
\end{equation}

By the Markov property of $\{S_{n}\}_{n}$,
\begin{align*}
E^x \left[\exp\left(t\sum_{i \equiv a \!\!\!\!\mod(M_n +1)}Y_{i, M_n} \right)\right]
&= E^x \left[ \prod_{i \equiv a \!\!\!\!\mod(M_n +1)} \exp(tY_{i, M_n})\right]\\
&\le \sup_{y \in G} E^y \left[\exp(tY_{0, M_n}) \right]^{n/(M_n +1)}\\
&= \left(1+ (\exp(t)-1)  \sup_{y \in G} P^y (T_{y} > M_n)  \right)^{n/(M_n +1)}.
\end{align*}

If $M_n$ is large, then, by Definition \ref{infsup}, 
\[ (\log M_n)  \sup_{y \in G} P^y (T_{y} > M_n)  \le r_{\sup} + \frac{\epsilon}{4} < +\infty. \]

For every $t \ge 0$,  we have that 
\[ E^x \left[\exp\left(t\sum_{i \equiv a \!\!\!\!\mod(M_n +1)}Y_{i, M_n} \right)\right]
\le \left(1+ (\exp(t)-1) \left(\frac{ r_{\sup} + \frac{\epsilon}{4}}{\log M_n} \right) \right)^{n/(M_n+1)}.\]

Hence, 
the right hand side of the inequality \eqref{sharp}  
is less than or equal to 
\begin{equation*} 
\left[ \exp\left(- \frac{t}{\log n}\left( r_{\sup} + \frac{\epsilon}{2}\right) \right)  \left(1+ (\exp(t)-1) \left(\frac{ r_{\sup} + \frac{\epsilon}{4}}{\log M_n} \right) \right) \right]^{n/(M_n +1)}. 
\end{equation*}

Let $t_1 > 0$ such that 
\[ \exp(t_1) \le 1 + \left(1+\frac{\epsilon}{20 r_{\sup}}\right) t_1. \]

Assume that $\epsilon \in (0, r_{\sup})$.  
If we let $M_n := n  / (\log n)^{2+\delta}$ for some $\delta > 0$, then, for large $n$, 
\[  1+ (\exp(t_1)-1)\left(\frac{ r_{\sup} + \frac{\epsilon}{4}}{\log M_n}\right) \le \exp\left(\frac{ r_{\sup}  + \frac{\epsilon}{4}}{\log M_n} \left(1+\frac{\epsilon}{20 r_{\sup}}\right) t_1\right) \le \exp\left(\frac{t_1}{\log n}\left( r_{\sup} + \frac{\epsilon}{3}\right) \right).  \] 

We have 
\[ (M_n +1)  \exp\left(- \frac{t_1 n}{M_n \log n}\frac{\epsilon}{6}\right) \]
\begin{equation}\label{decay-fast-1}
= \frac{n}{(\log n)^{2+\delta}} \exp\left(-  t_1 (\log n)^{1+\delta} \frac{\epsilon}{6} \right) \to 0, \ n \to \infty. 
\end{equation}
Thus assertion (i) holds. 
\end{proof}

\begin{proof}[Proof of Theorem \ref {WLLN-UC} (ii)]
First, 
\[ P^x \left(\frac{R_{n}}{n/\log n} < r_{\inf} - \epsilon \right) \le  P^x \left(\sum_{i=0}^{n-1} Y_{i, n-1} < \frac{n}{\log n} (r_{\inf} -  \epsilon)\right)\]
\begin{equation}\label{decompose}  
\le P^x \left(\sum_{i=0}^{n-1} Y_{i, M_n} < \frac{n}{\log n} \left(r_{\inf} -  \frac{\epsilon}{2}\right)\right) +  P^x \left(\sum_{i=0}^{n-1} Y_{i, M_n} - Y_{i, n-1} > \frac{n}{\log n} \frac{\epsilon}{2} \right)  \end{equation} 

Consider the first term of \eqref{decompose}. 
It holds that 
\[ P\left(\sum_{i=0}^{n-1} Y_{i, M_n} < \frac{n}{\log n} \left(r_{\inf} -  \frac{\epsilon}{2}\right)\right) \]
\[\le \sum_{a=0}^{M} P \left( \sum_{i \equiv a \!\!\!\!\mod(M_n + 1)}  Y_{i, M_n} < \frac{n/ \log n}{M+1}\left(r_{\inf} - \frac{\epsilon}{2}\right)\right). \]

For each $a$, 
\[ P^x \left(\sum_{i \equiv a \!\!\!\!\mod(M_n + 1)}Y_{i, M} < \frac{n/ \log n}{M_n + 1}\left(r_{\inf} - \frac{\epsilon}{2}\right)\right) \]
\[ \le \exp\left(t  \frac{n/ \log n}{M_n +1}\left(r_{\inf} - \frac{\epsilon}{2}\right) \right)
E^x \left[\exp\left( - t\sum_{i \equiv a \!\!\!\!\mod(M+1)}Y_{i, M_n} \right)\right] \]
\[ \le \exp\left(t  \frac{n/ \log n}{M_n +1}\left(r_{\inf} - \frac{\epsilon}{2}\right) \right)\left(1+ (\exp(-t)-1)  \inf_y P^y (T_{y} > M_n)  \right)^{n/(M_n +1)} \]
\[ \le \exp\left( \frac{n}{M_n +1} \left( t \frac{r_{\inf} - \frac{\epsilon}{2}}{\log n} - (1-\exp(-t)) \frac{r_{\inf} - \frac{\epsilon}{4}}{\log M} \right)   \right).\]
In the second inequality we have used the Markov property. 
In the third inequality we have used Definition \ref{infsup} and the inequality $1+x \le \exp(x)$. 

If $t > 0$ is sufficiently small, then, 
\[ t \left(r_{\inf} - \frac{\epsilon}{2}\right) - (1-\exp(-t)) \left(r_{\inf} - \frac{\epsilon}{4}\right) < 0. \]

As before, let $M_n := n / (\log n)^{2+\delta}$ for some $\delta > 0$. 
Then,  
\[ M_n P^x \left(\sum_{i \equiv a \!\!\!\!\mod(M_n + 1)}Y_{i, M_n} < \frac{n/ \log n}{M_n + 1}\left(r_{\inf} - \frac{\epsilon}{2}\right)\right) \]
\begin{equation}\label{fast-lower} 
\le n \exp\left( (\log n)^{1+\delta} \left( t \left(r_{\inf} - \frac{\epsilon}{2}\right) - (1-\exp(-t)) \left(r_{\inf} - \frac{\epsilon}{4}\right) \right)   \right) \to 0. 
\end{equation} 

Consider the second term of \eqref{decompose}. 
\begin{align*} 
P^x \left(\sum_{i=0}^{n-1} Y_{i, M_n} - Y_{i, n-1} > \frac{n \epsilon}{ 2\log n}  \right) 
&\le \frac{2 \log n}{\epsilon} \sup_{y \in V(G)} E^y \left[Y_{0, M_n} - Y_{0, n-1} \right] \\
&\le \frac{2 \log n}{\epsilon} \sup_{y \in V(G)} P^y (M_n < T_y \le n). 
\end{align*} 
By Lemma \ref{uc} and \eqref{uniform-error},  we have  that 
\[  (\log n) \sup_{y \in V(G)} P^y (M_n < T_y \le n) = O\left(\frac{\log\log n}{\log n}\right). \]
and hence 
\[ \lim_{n \to \infty} (\log n) \sup_{y \in V(G)} P^y (M_n < T_y \le n) = 0. \] 
By this, \eqref{fast-lower} and \eqref{decompose}, 
\begin{equation}\label{stronger}
 P^x \left(\frac{R_{n}}{n/\log n} < r_{\inf} - \epsilon \right)= O\left(\frac{\log\log n}{\log n}\right), 
 \end{equation}
and hence, the proof of (ii) is now completed. 
\end{proof}

\begin{proof}[Proof of  Theorem \ref{WLLN-UC} (iii)]
Assume $\alpha > 2$. 
Then, by the Markov property, 
\begin{align}\label{2nd-2}
P^x \left(\frac{R_n}{n/\log n} \le \frac{r_{\inf}}{\alpha}\right) 
&\le  P^x \left(\bigcap_{k = 1}^{2} \left\{ \left|  \left\{ S_{(k-1) \left\lfloor\frac{n}{2}\right\rfloor}, \dots, S_{k \left\lfloor \frac{n}{2} \right\rfloor} \right\} \right| \le  \frac{r_{\inf}}{\alpha}\frac{n}{\log n} \right\} \right) \notag\\
&\le  \sup_{y \in M} P^y \left(R_{\left\lfloor\frac{n}{2}\right\rfloor} \le  \frac{r_{\inf}}{\alpha}\frac{ n}{\log n} \right)^{2} \notag\\
&= O\left( \left( \frac{\log\log n}{\log n} \right)^2 \right),  
\end{align} 
where in the second inequality we have used the Markov property and in the last estimate we have used \eqref{stronger}.  

Hence
we have that for every $\eta > 0$, 
by the Borel-Cantelli lemma, 
we have that $P^x$-a.s., for sufficiently large $k$, 
\[ R_{\left\lfloor \exp( \eta k)  \right\rfloor} >  \frac{r_{\inf}}{\alpha} \frac{\exp(\eta k)}{ \eta k }. \]

If $\left\lfloor \exp( \eta k)  \right\rfloor \le n < \left\lfloor \exp( \eta (k+1))  \right\rfloor$, 
then, by using the assumption that $f$ is non-decreasing, 
\[ \frac{R_{n}}{n/\log n}  >  \frac{r_{\inf}}{\alpha} \frac{ R_{\left\lfloor \exp( \eta k)  \right\rfloor}}{\exp( \eta (k+1))}\eta k > \frac{r_{\inf}}{\alpha \exp(\eta)}.  \]

Hence, 
\[ \liminf_{n \to  \infty} \frac{R_{n}}{n/\log n}  > \frac{r_{\inf}}{\alpha \exp(\eta)}, \ \textup{ $P^x$-a.s.} \]
Since $\eta > 0$ is taken arbitrarily, 
 \[ \liminf_{n \to  \infty} \frac{R_{n}}{n/\log n}  \ge  \frac{r_{\inf}}{\alpha}, \ \textup{ $P^x$-a.s.} \]
 
Now we obtain the assertion by optimizing $\alpha$. 
\end{proof}

\begin{proof}[Proof of Theorem \ref{WLLN-UC} (iv)]
Let $\epsilon > 0$.  
By using that $R_n \le n$, we have that 
\[  \frac{E^x[R_n]}{n/\log n} \le (r_{\sup} + \epsilon)  P^x \left( \frac{R_n}{n/\log n} \le r_{\sup} + \epsilon \right) + (\log n) P^x \left( \frac{R_n}{n/\log n} > r_{\sup} + \epsilon \right). \]
By the estimate \eqref{decay-fast-1} in the final part of the proof of Theorem \ref {WLLN-UC} (i), 
\[ \lim_{n \to \infty} (\log n) P^x \left( \frac{R_n}{n/\log n} > r_{\sup} + \epsilon \right) = 0.  \]
Hence,
\[  \limsup_{n \to \infty} \frac{E^x[R_n]}{n/\log n} \le r_{\sup} + \epsilon. \]
By Theorem \ref {WLLN-UC} (ii), 
\[ \liminf_{n \to \infty} \frac{E^x[R_n]}{n/\log n} \ge (r_{\inf} - \epsilon) \liminf_{n \to \infty} P^x \left( \frac{R_n}{n/\log n} \ge r_{\inf} - \epsilon \right) = r_{\inf} - \epsilon. \]
Hence, 
\[  \liminf_{n \to \infty} \frac{E^x[R_n]}{n/\log n} \ge r_{\inf} - \epsilon. \]
Since $\epsilon$ is taken arbitrarily, we have the assertion. 
\end{proof} 


\section{Self-intersections of random walk on graph} 

In this section we consider  self-intersection of random walks.  
On the integer lattices, this object has been considered by  many papers. 
Our framework does not have spatial homogeneity and  the techniques are not applicable to our case, in particular, we need to distinguish self-intersection of a single random walk with intersections of two independent random walks starting at a common point.      
We deal with the latter case in Appendix. 
{\it Throughout this and the following sections, we always suppose the assumption of Theorem \ref{Lower}. } 

\begin{Prop}\label{self-intersect-main}
There exists a positive constant $C$ such that for every $\epsilon > 0$ and every $x \in V(G)$,
\[ P^{x}\left( \left| \{S_0, \dots, S_n\} \cap  \{S_n, \dots, S_{2n}\}  \right| \ge \frac{\epsilon n}{\log n}\right) \le \frac{C}{\epsilon^2 (\log n)^{3/2}} \]
\end{Prop}

As in \cite{L}, the estimate in the following proposition is attributed to obtaining an upper bound for the expectation of the intersection. 

\begin{Prop}\label{self-intersect-expectation}
There exists a positive constant $C$ such that for every $x \in V(G)$ and $n \ge 3$, 
\[  E^{x}\left[ \left| \{S_0, \dots, S_n\} \cap  \{S_n, \dots, S_{2n}\}   \right|^2 \right] \le C \frac{n^2 (\log\log n)^2}{(\log n)^4}. \]
\end{Prop}

For the case that $G = \Z^2$, we have a better estimate (see \cite[Theorem 5.1]{L}).   
We conjecture that $(\log\log n)^2$ in the numerator of the right hand side of the above inequality could be removed.   
We use estimates for hitting time distributions on graphs (cf. \cite{B17}). 

\begin{Lem}\label{hitting-self}
There exists a positive constant $C$ such that 
for every $x, y$ such that $d(x,y) \ge \left(\dfrac{n}{(\log n)^4}\right)^{1/d}$, 
\[ P^x (T_y \le n) \le C \frac{\log\log n}{\log n}. \]
\end{Lem}

This lemma plays a role similar to  \cite[Theorem 3.6]{L}. 
We define  {\it  the Green function killed on a finite subset} $A$ of $V(G)$ by 
\[ G_A (x,y) = E^x \left[ \sum_{i=0}^{T_{A^c}} 1_{\{S_i = y\}} \right].  \]
The assumptions of Theorem \ref{Lower} are used in the proof. 
 
\begin{proof}
If $d(x,y) > n$, then, $P^x (T_y \le n) = 0$. 
Hereafter we assume $d(x,y) \le n$. 

Case 1. $n \ge d(x,y) \ge n^{1/d}(\log n)^2$. 
By \eqref{sGHK-upper} in the assumptions of Theorem \ref{Lower}, we have that there exist positive constants $c_1, c_2$ and $c_3$ such that for every $n \ge 1$ and $x, y \in V(G)$ satisfying that $d(x,y) \le n$, 
\[ P^x (T_y \le n) = \sum_{i = d(x,y)}^{n} P^x(T_y = i) \le  \sum_{i = d(x,y)}^{n} p_i (x,y)  \]
\[ \le \sum_{i = d(x,y)}^{n} \frac{c_1}{i}  \exp\left(-c_2 \frac{d(x,y)^{d/(d-1)}}{i^{1/(d-1)}}\right) \]
\[ \le  \exp\left(-c_2 \frac{(n^{1/d}(\log n)^2)^{d/(d-1)}}{n^{1/(d-1)}}\right) \sum_{i = d(x,y)}^{n} \frac{c_1}{i} \]
\[ \le c_3 \log \left(\frac{n}{n^{1/d}(\log n)^2}\right) \exp(-c_2 (\log n)^{2d/(d-1)}). \]
The upper bound is smaller than $O\left( \frac{\log\log n}{\log n} \right)$.

Case 2. $\left(\dfrac{n}{(\log n)^4}\right)^{1/d} \le d(x,y) < n^{1/d}(\log n)^2$. 
\[  P^x (T_y \le n) \le  P^x (T_{B(x, 3n^{1/d}(\log n)^2)^c} \le n) + P^x (T_y \le T_{B(x, 3n^{1/d}(\log n)^2)^c}).  \]

By \cite[Lemma 3.1]{KuN} (or \cite[Lemma 4.21(a)]{B17}), 
there exist positive constants $c_4$ and $c_5$ such that or every $n \ge 1$ and $x, y \in V(G)$ satisfying that $\left(\dfrac{n}{(\log n)^2}\right)^{1/d} \le d(x,y) < n^{1/d}(\log n)^2$, 
\[ P^x (T_{B(x, 3n^{1/d}(\log n)^2)^c} \le n) \le c_4 \exp\left(-c_5 (\log n)^{2d/(d-1)}\right). \]

By \cite[Theorem 1.31 and Theorem 4.26(c)]{B17}, 
there exists a positive constant $c_6$ such that for every $n \ge 1$ and $x, y \in V(G)$ satisfying that $\left(\dfrac{n}{(\log n)^2}\right)^{1/d} \le d(x,y) < n^{1/d}(\log n)^2$, 
\[  P^x (T_y \le T_{B(x, 3n^{1/d}(\log n)^2)^c}) = \frac{G_{B(x, 3n^{1/d}(\log n)^2)} (x,y)}{G_{B(x, 3n^{1/d}(\log n)^2)} (y,y)}  \]
\[ \le c_6 \frac{\log \left(3n^{1/d}(\log n)^2 / d(x,y)\right)}{\log (3n^{1/d}(\log n)^2)} \]
\[ \le c_6 \frac{\log \left(3n^{1/d}(\log n)^2 / \left(\dfrac{n}{(\log n)^2}\right)^{1/d} \right)}{\log (3n^{1/d}(\log n)^2)} = O\left(\frac{\log\log n}{\log n}\right). \]
Thus we have Lemma \ref{hitting-self}. 
\end{proof}

\begin{proof}[Proof of Proposition \ref{self-intersect-expectation}]
In this proof we let 
\[ T^{0}_y := \inf\{n \ge 0 : S_n = y\}, \ y \in V(G). \]
By the Markov property, we have that 
 \[ E^{x}\left[ \left| \{S_0, \dots, S_n\} \cap  \{S_n, \dots, S_{2n}\}    \right|^2 \right] = E^{x}\left[ \left(\sum_{y}  1_{\{  y \in \{S_0, \dots, S_n\}\}}1_{y \in  \{S_n, \dots, S_{2n}\}\}} \right)^2 \right]\]
 \[ = \sum_{y_1, y_2}  E^{x}\left[ 1_{\{  y_1,  y_2 \in \{S_0, \dots, S_n\}\}}  1_{ y_1,  y_2 \in  \{S_n, \dots, S_{2n}\}\}} \right]    \]
\begin{equation}\label{self-intersect-sum}
=  \sum_{y_1, y_2, z}  P^{x}\left(  T^{0}_{y_1} \le n,  T^{0}_{y_2} \le n, S_n = z\right)  P^{z} \left(   T^{0}_{y_1} \le n,  T^{0}_{y_2} \le n  \right). 
\end{equation}

We first consider that case that $y_1 = y_2$ in \eqref{self-intersect-sum}. 
By  \eqref{Ahlfors} and Lemma \ref{hitting-self},   
\[ \sum_{y, z}  P^{x}\left(  T^{0}_{y} \le n,  S_n = z\right)  P^{z} \left(   T^{0}_{y} \le n  \right) \le \sum_{y,z; d(y,z) \le \left(\frac{n}{(\log n)^4}\right)^{1/d}} P^{x}\left( T^{0}_{y} \le n,  S_n = z\right)  \]
\[ +  \sum_{y,z; d(y,z) > \left(\frac{n}{(\log n)^4}\right)^{1/d}} P^{x}\left(  T^{0}_{y} \le n,  S_n = z\right) \cdot \frac{C \log\log n}{\log n}.  \]
\[ \le  \sum_{z} \sum_{y; d(y,z) \le \left(\frac{n}{(\log n)^4}\right)^{1/d}} P^{x}\left(S_n = z\right)  +  \sum_{y} P^{x}\left(  T^{0}_{y} \le n\right) \cdot \frac{C \log\log n}{\log n} \]
\begin{equation}\label{seif-inter-moment-1-pre} 
\le C \left(\frac{n}{(\log n)^2} + E^x [R_n] \frac{\log\log n}{\log n} \right). 
\end{equation} 

By Theorem \ref{WLLN-UC} (iv), 
\begin{equation}\label{exp-upper}
\sup_{x \in V(G)} E^x[R_n] = O\left(\frac{n}{\log n}\right). 
\end{equation}
By substituting this into \eqref{seif-inter-moment-1-pre}, 
\begin{equation}\label{seif-inter-moment-1} 
\sum_{y, z}  P^{x}\left(  T^{0}_{y} \le n,  S_n = z\right)  P^{z} \left(   T^{0}_{y} \le n  \right) = O\left(\frac{n \log\log n}{(\log n)^2}\right). 
\end{equation}  

By symmetry, we have that for each fixed $z$, 
\[  \sum_{y_1, y_2; y_1 \ne y_2}  P^{x}\left(  T^{0}_{y_1} \le n,  T^{0}_{y_2} \le n, S_n = z\right)  P^{z} \left(   T^{0}_{y_1} \le n,  T^{0}_{y_2} \le n  \right) \]
\[ = 2 \sum_{y_1, y_2; y_1 \ne y_2}  P^{x}\left(  T^{0}_{y_1} \le n,  T^{0}_{y_2} \le n, S_n = z\right)  P^{z} \left(   T^{0}_{y_1} < T^{0}_{y_2} \le n  \right). \]
Hence, by taking the sum with respect to $z$, 

\[ \sum_{y_1, y_2, z; y_1 \ne y_2}  P^{x}\left(  T^{0}_{y_1} \le n,  T^{0}_{y_2} \le n, S_n = z\right)  P^{z} \left(   T^{0}_{y_1} \le n,  T^{0}_{y_2} \le n  \right) \]
\begin{equation}\label{symm}  
= 2 \sum_{y_1, y_2, z; y_1 \ne y_2}  P^{x}\left(  T^{0}_{y_1} \le n,  T^{0}_{y_2} \le n, S_n = z\right)  P^{z} \left(   T^{0}_{y_1} < T^{0}_{y_2} \le n  \right). 
\end{equation}

Consider the following four sets.  
For each fixed $z$, let 
\[ A_1 (z) := \left\{(y_1, y_2) : d(y_1,z) \le \left(\frac{n}{(\log n)^4}\right)^{1/d}, d(y_1,y_2) \le \left(\frac{n}{(\log n)^4}\right)^{1/d} \right\}, \]
\[ A_2 (z)  := \left\{(y_1, y_2) : d(y_1,z) \le \left(\frac{n}{(\log n)^4}\right)^{1/d}, d(y_1,y_2) > \left(\frac{n}{(\log n)^4}\right)^{1/d} \right\}, \]
\[ A_3 (z)  := \left\{(y_1, y_2) : d(y_1,z) > \left(\frac{n}{(\log n)^4}\right)^{1/d}, d(y_1,y_2) \le \left(\frac{n}{(\log n)^4}\right)^{1/d} \right\}, \]
and 
\[ A_4 (z)  := \left\{(y_1, y_2) : d(y_1,z) > \left(\frac{n}{(\log n)^4}\right)^{1/d}, d(y_1,y_2) > \left(\frac{n}{(\log n)^4}\right)^{1/d} \right\}. \]

(1) By  \eqref{Ahlfors}, we have that 
\[ \sum_z \sum_{(y_1, y_2) \in A_1(z) ; y_1 \ne y_2}  P^{x}\left(  T^{0}_{y_1} \le n,  T^{0}_{y_2} \le n, S_n = z\right)  P^{z} \left(   T^{0}_{y_1} < T^{0}_{y_2} \le n  \right)\]
\[ \le \sum_{z} \sum_{y_1, y_2; d(y_1,z) \le \left(\frac{n}{(\log n)^4}\right)^{1/d},  d(z,y_2) \le 2\left(\frac{n}{(\log n)^4}\right)^{1/d}}  P^{x}\left(S_n = z\right) 
= O\left(\frac{n^2}{(\log n)^8}\right). \]

(2) By the Markov property, Lemma \ref{hitting-self} and \eqref{exp-upper}, 
\[ \sum_z \sum_{(y_1, y_2) \in A_2(z); y_1 \ne y_2}  P^{x}\left(  T^{0}_{y_1} \le n,  T^{0}_{y_2} \le n, S_n = z\right)  P^{z} \left(   T^{0}_{y_1} < T^{0}_{y_2} \le n  \right)\]
\[ \le \sum_z \sum_{(y_1, y_2) \in A_2(z); y_1 \ne y_2}  P^{x}\left(  T^{0}_{y_1} \le n,  T^{0}_{y_2} \le n, S_n = z\right)  P^{z}(   T^{0}_{y_1} \le n)P^{y_1}( T^{0}_{y_2} \le n ) \] 
\[ \le \sum_z \sum_{(y_1, y_2) \in A_2 (z); y_1 \ne y_2}  P^{x}\left(  T^{0}_{y_1} \le n,  T^{0}_{y_2} \le n, S_n = z\right)  P^{z}(   T^{0}_{y_1} \le n) \cdot \frac{C \log\log n}{\log n}. \] 
\[ \le  \frac{C \log\log n}{\log n} \sum_z   \sum_{y_2} \left(P^{x}\left( T^{0}_{y_2} \le n, S_n = z\right)  \sum_{y_1; d(y_1,z) \le \left(\frac{n}{(\log n)^4}\right)^{1/d}} P^{z}(   T^{0}_{y_1} \le n)\right) \]
\[ \le  \frac{C n\log\log n}{(\log n)^5} \sum_z   \sum_{y_2}    P^{x}\left( T^{0}_{y_2} \le n, S_n = z\right)  \]
\[ = \frac{C n\log\log n}{(\log n)^5} \sum_{y_2}    P^{x}\left( T^{0}_{y_2} \le n\right) = O\left(\frac{n^2 \log\log n}{(\log n)^6} \right). \]

(3) By the Markov property, and \eqref{seif-inter-moment-1}, 
\[ \sum_z \sum_{(y_1, y_2) \in A_3 (z); y_1 \ne y_2}  P^{x}\left(  T^{0}_{y_1} \le n,  T^{0}_{y_2} \le n, S_n = z\right)  P^{z} \left(   T^{0}_{y_1} < T^{0}_{y_2} \le n  \right)\]
\[ \le \sum_z \sum_{(y_1, y_2) \in A_3 (z); y_1 \ne y_2}  P^{x}\left(  T^{0}_{y_1} \le n,  T^{0}_{y_2} \le n, S_n = z\right)  P^{z}(   T^{0}_{y_1} \le n)P^{y_1}( T^{0}_{y_2} \le n ) \] 
\[ \le \sum_z  \sum_{y_1; d(y_1,z) > \left(\frac{n}{(\log n)^4}\right)^{1/d}}  \left(P^{x}\left(  T^{0}_{y_1} \le n, S_n = z\right)  P^{z}(   T^{0}_{y_1} \le n) \sum_{y_2; d(y_1,y_2) \le \left(\frac{n}{(\log n)^4}\right)^{1/d}} P^{y_1}( T^{0}_{y_2} \le n ) \right)  \] 
\[ \le \frac{Cn}{(\log n)^4} \sum_z \sum_{y_1; d(y_1,z) > \left(\frac{n}{(\log n)^4}\right)^{1/d}} P^{x}\left(  T^{0}_{y_1} \le n, S_n = z\right)  P^{z}(   T^{0}_{y_1} \le n) \]
\[ \le \frac{Cn \log\log n}{(\log n)^5} \sum_z \sum_{y_1} P^{x}\left(  T^{0}_{y_1} \le n, S_n = z\right)  \]
\[ \le \frac{Cn \log\log n}{(\log n)^5} \sum_{y_1} P^{x}\left(  T^{0}_{y_1} \le n\right)  = O\left(\frac{n^2 \log\log n}{(\log n)^6} \right). \]

(4) By the Markov property, Lemma \ref{hitting-self} and \eqref{exp-upper}, 
\[ \sum_z \sum_{(y_1, y_2) \in A_4 (z); y_1 \ne y_2}  P^{x}\left(  T^{0}_{y_1} \le n,  T^{0}_{y_2} \le n, S_n = z\right)  P^{z} \left(   T^{0}_{y_1} < T^{0}_{y_2} \le n  \right)\]
\[ \le \sum_z \sum_{(y_1, y_2) \in A_4(z); y_1 \ne y_2}  P^{x}\left(  T^{0}_{y_1} \le n,  T^{0}_{y_2} \le n, S_n = z\right)  P^{z}(   T^{0}_{y_1} \le n)P^{y_1}( T^{0}_{y_2} \le n ) \] 
\[ \le \frac{C(\log\log n)^2}{(\log n)^2} \sum_z \sum_{y_1 \ne y_2}  P^{x}\left(  T^{0}_{y_1} \le n,  T^{0}_{y_2} \le n, S_n = z \right) \] 
\[ \le \frac{C(\log\log n)^2}{(\log n)^2} \sum_{y_1 \ne y_2}  P^{x}\left(  T^{0}_{y_1} \le n,  T^{0}_{y_2} \le n \right) \] 
\[ \le \frac{2C(\log\log n)^2}{(\log n)^2} \sum_{y_1}  P^{x}(  T^{0}_{y_1} \le n) \left(\sum_{y_2} P^{y_1} ( T^{0}_{y_2} \le n )\right) \] 
\[ \le \frac{2C(\log\log n)^2}{(\log n)^2} \left( \sup_z E^z [R_n] \right)^2 \]
\[ = O\left(\frac{n^2 (\log\log n)^2}{(\log n)^4}\right), \]
where in the fourth inequality we have used that  
\[ \sum_{y_1 \ne y_2}  P^{x}\left(  T^{0}_{y_1} \le n,  T^{0}_{y_2} \le n \right) \le 2 \sum_{y_1, y_2}  P^{x}\left(  T^{0}_{y_1} < T^{0}_{y_2} \le n \right) \le 2  \sum_{y_1}  P^{x}(  T^{0}_{y_1} \le n) \left(\sum_{y_2} P^{y_1} ( T^{0}_{y_2} \le n )\right). \]

By (1) - (4), we have that 
\[ \sum_{y_1, y_2, z; y_1 \ne y_2}  P^{x}\left(  T^{0}_{y_1} \le n,  T^{0}_{y_2} \le n, S_n = z\right)  P^{z} \left(   T^{0}_{y_1} < T^{0}_{y_2} \le n  \right) = O\left(\frac{n^2 (\log\log n)^2}{(\log n)^4}\right). \]

This, \eqref{symm} and \eqref{seif-inter-moment-1} imply the assertion. 
\end{proof} 

\begin{Rem}
(i) We conjecture that for every natural number $p$, there exists a positive constant $C$ such that for every $x \in V(G)$ and $n \ge 3$, 
\[  E^{x}\left[ \left| \{S_0, \dots, S_n\} \cap  \{S_n, \dots, S_{2n}\}   \right|^p \right] \le C \frac{n^p (\log\log n)^p}{(\log n)^{2p}}. \]
(ii) In \cite[Lemma 3.1]{KuN}, the Ahlfors regularity is assumed. 
However,  we only need the upper bound for the polynomial volume growth. 
See the proof of \cite[Lemma 4.21(a)]{B17}. \\
(iii) In \cite[Theorem 4.26(c)]{B17}, the sub-Gaussian heat kernel upper and lower bounds are both assumed. 
However, we only need the upper bound for $G_{B(x, 3n^{1/d}(\log n)^2)} (x,y)$ and the lower bound for $G_{B(x, 3n^{1/d}(\log n)^2)} (x,x)$. 
In order to establish these two bounds, we need the sub-Gaussian heat kernel upper bound and the lower bound for the on-diagonal  heat kernel lower bound, respectively. \\
(iv) The inductive technique in the proof of \cite[Lemme 6.2]{L} does not work well in our case at least in direct manners. 
Our framework cannot induce the numerics ``$2^{1/2}$" in the proof of \cite[Lemme 6.2]{L}. 
Instead of using variance estimates, we develop a new proof of the lower bound of the strong law of large numbers in the following section.  
\end{Rem} 


\section{Proof of Theorem \ref{Lower}}

We will show that for every $\epsilon > 0$, 
\begin{equation}\label{2nd-full}
\sup_x P^x \left(\frac{R_n}{n/\log n} \le \frac{r_{\inf}}{1+\epsilon}\right) = O\left( \left( \frac{1}{\log n} \right)^{3/2} \right)
\end{equation}
Once this is established, then, we can derive Theorem \ref{Lower} by adopting the same interpolation argument as in the proof of Theorem \ref{WLLN-UC} (iii). 

We will show \eqref{2nd-full} by a form of induction with respect to the coefficient of $r_{\inf}$. 
More specifically, we will show that there exists a strictly increasing {\it finite} sequence $(a_0, \dots, a_k)$ such that \\
(i) $a_0 = \epsilon$\\
(ii) $1+ a_k > 2$\\
(iii) For every $0 \le i \le k$, 
\begin{equation}\label{2nd-induction}
P^x \left(\frac{R_n}{n/\log n} \le \frac{r_{\inf}}{1 + a_i}\right) =O\left(  \left( \frac{1}{\log n} \right)^{3/2}\right)
\end{equation}

If $i = k$, \eqref{2nd-induction} is established by \eqref{2nd-2} in the proof of Theorem \ref{WLLN-UC} (iii).
We will show \eqref{2nd-induction} by induction in $i$ in reversed order and finally obtain \eqref{2nd-full}.

Let $\epsilon_0 > 0$.
Let $\epsilon_1$ and $\epsilon_2$ be two positive real number smaller than $\epsilon_0$ which will be specified later.  
For each $n \ge 4$, we consider the following four events. 
\[ A := \left\{\left|\{S_0, \dots, S_{\left\lfloor n/2 \right\rfloor}\} \cap \{S_{\left\lfloor n/2 \right\rfloor}, \dots, S_n\} \right| > \dfrac{\epsilon_1 n}{\log n} \right\}. \]
\[ B := \left\{\left|\{S_0, \dots, S_{\left\lfloor n/2 \right\rfloor}\} \right| \le \dfrac{\left\lfloor n/2 \right\rfloor}{\log n} \dfrac{r_{\inf}}{1+\epsilon_2}, \left| \{S_{\left\lfloor n/2 \right\rfloor}, \dots, S_n\} \right| \le \dfrac{\left\lfloor n/2 \right\rfloor}{\log n} \dfrac{r_{\inf}}{1+\epsilon_2}\right\}. \] 
\[ C_1 := \left\{ \left| \{S_{\left\lfloor n/2 \right\rfloor}, \dots, S_n\} \right| > \dfrac{\left\lfloor n/2 \right\rfloor}{\log n} \dfrac{r_{\inf}}{1+\epsilon_2}\right\}. \] 
\[ C_2 := \left\{ \left|\{S_0, \dots, S_{\left\lfloor n/2 \right\rfloor}\} \right| > \dfrac{\left\lfloor n/2 \right\rfloor}{\log n} \dfrac{r_{\inf}}{1+\epsilon_2} \right\}. \] 

Then, 
\[ \left\{ \frac{R_n}{n/\log n} \le \frac{r_{\inf}}{1+\epsilon_0} \right\} \]
\[\subset A \cup B \cup \left(C_1 \cap A^c \cap \left\{ \frac{R_n}{n/\log n} \le \frac{r_{\inf}}{1+\epsilon_0} \right\} \right) \cup \left(C_2 \cap A^c \cap \left\{ \frac{R_n}{n/\log n} \le \frac{r_{\inf}}{1+\epsilon_0} \right\} \right). \]

For ease of notation, we temporarily assume $n$ is an even number. 

We estimate $P^x (A)$. 
By using the assumption that $G$ has the bounded degree and Proposition \ref{self-intersect-main}, 
it holds that there exists a positive constant $C$ such that for every $x$, 
\[ P^x (A) = P^{x}\left( \left| \{S_0, \dots, S_{\left\lfloor n/2 \right\rfloor}\} \cap  \{S_{\left\lfloor n/2 \right\rfloor}, \dots, S_n\}  \right| \ge \frac{\epsilon_1 n}{\log n}\right) \le \frac{C  \max_{x \in V(G)} \deg(x) }{\epsilon_1^4 (\log n)^{3/2}}. \]

We estimate $P^x (B)$ in the same manner as in the proof of Theorem \ref{WLLN-UC} (iii). 
By the Markov property we have that 
\[ P^x (B) \le  \sup_y P^y \left( R_{\left\lfloor n/2 \right\rfloor } \le \frac{r_{\inf}}{1+\epsilon_2 / 2} \frac{\left\lfloor n/2 \right\rfloor }{\log \left\lfloor n/2 \right\rfloor } \right)^2 =  O\left(\left( \frac{\log\log n}{\log n} \right)^{2}\right) \]

We estimate $P^x \left(C_i \cap A^c \cap \left\{ \frac{R_n}{n/\log n} \le \frac{r_{\inf}}{1+\epsilon_0} \right\}\right), i = 1,2$. 
If $\epsilon_1$ and $\epsilon_2$ are sufficiently small, 
then, 
\[ C_1  \cap A^c \cap \left\{ \frac{R_n}{n/\log n} \le \frac{r_{\inf}}{1+\epsilon_0} \right\} \subset \left\{\left|\{S_0, \dots, S_{\left\lfloor n/2 \right\rfloor}\} \right| \le \dfrac{n/2}{\log n} \dfrac{r_{\inf}}{1+ 3\epsilon_0/2}\right\}, \]
and, 
\[ C_2  \cap A^c \cap \left\{ \frac{R_n}{n/\log n} \le \frac{r_{\inf}}{1+\epsilon_0} \right\} \subset  \left\{\left|\{S_{\left\lfloor n/2 \right\rfloor}, \dots, S_{n}\} \right| \le \dfrac{n/2}{\log n} \dfrac{r_{\inf}}{1+ 3\epsilon_0/2}\right\}.  \]

Hence we have that for $i = 1,2$ and for sufficiently large $n$
\[ P^x \left(C_i \cap A^c \cap \left\{ \frac{R_n}{n/\log n} \le \frac{r_{\inf}}{1+\epsilon_0} \right\}\right) \le \sup_y P^y \left( R_{\left\lfloor n/2 \right\rfloor} \le \frac{r_{\inf}}{1+ 4\epsilon_0/3} \frac{\left\lfloor n/2 \right\rfloor}{\log \left\lfloor n/2 \right\rfloor}  \right). \]

Thus we have that 
\[ \sup_x P^x \left(R_n \le \frac{r_{\inf}}{1+\epsilon_0}\frac{n}{\log n}\right) \le 2 \sup_y P^y \left(R_{n/2}  \le \frac{r_{\inf}}{1+ 4\epsilon_0/3} \frac{\left\lfloor n/2 \right\rfloor}{\log \left\lfloor n/2 \right\rfloor}\right)  + O\left(\left( \frac{1}{\log n} \right)^{3/2}\right). \]

This completes the inductive step in reversed order. 
Indeed, we have now that  if \eqref{2nd-induction} also holds for $a_{i} = 4\epsilon_0 /3$, then \eqref{2nd-induction} holds for $a_{i-1} = \epsilon_0$.  

Let $\epsilon > 0$. 
Let $k = k(\epsilon)$ be a natural number such that $(4/3)^k \epsilon > 1$. 
Let $a_0, \dots, a_k$ be a sequence such that $a_0 = \epsilon$ and $a_i = (4/3)^i \epsilon$ for every $i$. 
Then by repeating the above argument, 
we see that for every sufficiently large $n$ which is a multiple of $2^k$, 
\[ \sup_x P^x \left(\frac{R_n}{n/\log n} \le \frac{r_{\inf}}{1+\epsilon}\right) \le 2^k \sup_x P^x \left(\frac{R_n}{n/\log n} \le \frac{r_{\inf}}{1+a_k}\right) + O\left(\left( \frac{1}{\log n} \right)^{3/2}\right). \]

Hence \eqref{2nd-full} holds for sufficiently large $n$ which is a multiple of $2^k$. 
By the interpolation argument as in the proof of Theorem \ref{WLLN-UC} (iii), 
we see that \eqref{2nd-full} holds for sufficiently large $n$.


\section{Proof of Theorem \ref{SLLN}}

\begin{proof}[Proof of Theorem \ref{SLLN}]
By Theorem \ref{WLLN-UC} (i), 
\[ \limsup_{n \to \infty} \frac{R_n}{n/\log n} \in [0, r_{\sup}], \ P^x\textup{-a.s.}\]
The random variables $\displaystyle \limsup_{n \to \infty} \frac{R_n}{n/\log n}$ and $\displaystyle \liminf_{n \to \infty} \frac{R_n}{n/\log n}$ are both in the tail $\sigma$-algebra\\ 
$$\bigcap_{n \ge 1} \sigma\left(\{S_k : k \ge n\}\right).$$ 
Since  $G$ is Ahlfors $d$-regular and satisfies the sub-Gaussian heat kernel estimates with walk dimension $d$ for some $d > 1$,  
the zero-one law by Barlow-Bass \cite[Theorem 8.4]{BB99} holds. 
Hence, there exist non-random constants $c_{\inf}, c_{\sup} \in [0, r_{\sup}]$ such that 
\[ \liminf_{n \to \infty} \frac{R_n}{n/\log n} = c_{\inf}, \ \textup{ $P^x$-a.s.}\] 
\[ \limsup_{n \to \infty} \frac{R_n}{n / \log n} = c_{\sup}, \ \textup{ $P^x$-a.s.}  \]
We remark that neither $c_{\inf}$ or $c_{\sup}$ depends on $x$. 
\eqref{c-inf-Fatou-new} and \eqref{sandwich-new} now follow from Fatou's lemma and Theorem \ref{WLLN-UC}. 
\end{proof} 

\begin{Rem}
We are not sure whether the following holds or not:\\ 
$c_{\sup} = r_{\sup}$, $c_{\inf} = r_{\inf}$,  
\[ c_{\sup} = \limsup_{n \to \infty} \frac{E^x[R_n]}{n / \log n}, \]
\[ c_{\inf} = \liminf_{n \to \infty} \frac{E^x[R_n]}{n / \log n}. \]
\end{Rem} 


\section{Proofs of Theorems \ref{fm-2d} and \ref{fluc-2d}} 

We adopt a strategy similar to the one used to show the corresponding result for the Wiener sausage in the metric measure space with spectral dimension two (\cite[Theorem 1.4 (ii)]{O-ws}).

We first remark that by using the Markov property it is easy to see that 
\begin{Lem}\label{ws}
For every $x \ne y \in G$ and every $n, k \ge 1$,  we have that 
\begin{equation}\label{2d-range-hit} 
E_{G}^x [R_n] = 1+\sum_{y \in G, y \ne x} P^x (T_y \le n), 
\end{equation} 
\begin{equation}\label{2d-hit-upper} 
P^x (T_y \le n) \le \frac{\sum_{i \le n+k} P^x (S_i = y)}{\sum_{i \le k} P^y (S_i = y)},  
\end{equation} 
and, 
\begin{equation}\label{2d-hit-lower} 
P^x (T_y \le n) \ge \frac{\sum_{i \le n} P^x (S_i = y)}{\sum_{i \le n} P^y (S_i = y)}. 
\end{equation} 
\end{Lem}

Let $G$ be a finite modification of $\Z^2$.  
Let $o \in G$. 
If we let $N_0 > \max_{x \in D} d_G (o, x)$, then, $D \subset B_G (o, N_0)$. 

\begin{proof}[Proof of Theorem \ref{fm-2d}] 
Let $\delta > 0$ and $m \ge 1$. 
Let 
\[ k_{n} := \delta \frac{n^{1/2}}{\log n}, \textup{  \ and  \ } l_{m, n} := \frac{n^{m/(m+1)}}{(\log n)^3}. \] 
We take the integer parts of the quantities of the right hand sides if it is needed.

If $y \in G \setminus B_{G}(o, 2k_{n})$, then, by the definition of the finite modification, 
$B_G (o, k_n)$ is graph-isomorphic to $B_{\Z^2}(0, k_n)$ for large $k_n$. 
Hence, 
\[ P_G^y \left(S_i = y, T_{B_G (y, k_n)^c} > i \right) = P^{0}_{\Z^2} (S_i = 0, T_{B_{\Z^2} (0, k_n)^c} > i ), \]
and, 
\[  P_G^y \left(T_{B_G (y, k_n)^c} \le i \right) = P^{0}_{\Z^2} (T_{B_{\Z^2} (0, k_n)^c} \le i ). \]
By using them, 
\begin{align*} 
\left|\sum_{i \le l_{m, n}} P_{G}^y (S_i = y) - P_{\Z^2}^0 (S_i = 0)\right|  &\le \sum_{i \le l_{m, n}} P_{\Z^2}^0 \left(T_{B(0,k_n)^c} \le i \right) \\
&= \sum_{i \le l_{m, n}} P_{\Z^2}^0 \left(\max_{j \le i} d_{\Z^2}(0, S_j) \ge k_n\right). 
\end{align*} 

By the Burkholder-Davis-Gundy inequality for degree $m \ge 2$, 
\[ P_{\Z^2}^0 \left(\max_{j \le i} d_{\Z^2}(0, S_j) \ge k_n\right) \le \frac{1}{k_n^{2m}}  E_{\Z^2}^0 \left[\max_{j \le i} d_{\Z^2}(0, S_j)^{2m} \right] \le C_m \frac{i^{m}}{k_n^{2m}}.  \]

Therefore, 
\begin{equation}\label{2d-fm-core}
\left|\sum_{i \le l_{m, n}} P_{G}^y (S_i = y) - P_{\Z^2}^0 (S_i = 0)\right| \le \frac{C_m}{(\log n)^{m+3}}.  
\end{equation}    

First we give the upper estimate. 
By \eqref{2d-range-hit}, \eqref{2d-hit-upper} and \eqref{2d-fm-core}, 
we have  that  
\begin{align*} 
\sum_{y \in G \setminus B_{G}(o, 2k_{n})} P_{G}^o (T_y \le n) 
&\le \sum_{y \in G \setminus B_{G}(o, 2k_{n})} \frac{\sum_{i \le n + l_{m, n}} P_{G}^x (S_i = y)}{\sum_{i \le l_{m, n}} P_{G}^y (S_i = y)} \\
&\le (n + l_{m, n})\left(\frac{1}{\sum_{i \le l_{m, n}}  P_{\Z^2}^0 (S_i = 0)}  + \frac{C_m}{(\log n)^{m+3}}\right). 
\end{align*}  

Hence, 
\[ \limsup_{n \to \infty} \frac{\log n}{n} \sum_{y \in G \setminus B_{G}(o, 2k_{n})} P_{G}^o (T_y \le n) \le \pi.  \]

On the other hand, by recalling  the definition of $k_n$ and the rate of the volume growth of $G$,   
\[ \limsup_{n \to \infty} \frac{\log n}{n} \sum_{y \in B_{G}(o, 2k_{n})} P_{G}^o (T_y \le n) \le C \delta^2. \]
Since $\delta > 0$ is taken arbitrarily, 
we have  that 
\[ \limsup_{n \to \infty} \frac{\log n}{n} E^{o}_{G} [R_n] \le \pi. \] 

Second we give the lower estimate. 
It holds that $G$ is roughly isometric to $\Z^2$. 
Since the isoperimetric inequality of order 2, which is denoted by $(IS_{2})$ here, is stable with respect to rough isometry, 
$(IS_{2})$, and hence the Nash inequality of order $2$, holds also for $G$.  (See \cite[Theorem 4.7 and Corollary 14.5]{W})

By using \eqref{2d-fm-core} and the Nash inequality,  
\[ \sup_{y \in G \setminus B_{G}(o, 2k_{n})} \left|\sum_{i \le n} P_{G}^y (S_i = y) - P_{\Z^2}^0 (S_i = 0) \right| 
\le \frac{C_m}{(\log n)^{m+3}} +  \frac{C \log n}{m+1}.  \]
Hence, by letting $n \to \infty$ first, and then $m \to \infty$,  
\[ \limsup_{n \to \infty} \frac{1}{\log n} \sup_{y \in G \setminus B_{G}(o, 2k_{n})} 
\left| \sum_{i \le n} P_{G}^y (S_i = y) - P_{\Z^2}^0 (S_i = 0) \right| = 0.  \]
By this, \eqref{2d-range-hit} and \eqref{2d-hit-lower}, 
we have  that  
\begin{align*} 
\sum_{y \in G \setminus B_{G}(o, 2k_{n})} P_{G}^o (T_y \le n) 
&\ge \sum_{y \in G \setminus B_{G}(o, 2k_{n})}  \frac{\sum_{i \le n} P_{G}^x (S_i = y)}{\sum_{i \le n} P_{G}^y (S_i = y)} \\
&\ge \frac{\sum_{y \in G \setminus B_{G}(o, 2k_{n})} \sum_{i \le n} P_{G}^x (S_i = y)}{\sum_{i \le n} P_{G}^0 (S_i = 0) + o(1) \log n}. 
\end{align*}
Hence, by recalling the definition of $k_n$, 
\[ \liminf_{n \to \infty} \frac{\log n}{n} \sum_{y \in G \setminus B_{G}(o, 2k_{n})} P_{G}^o (T_y \le n) \ge \pi. \] 
\end{proof} 


Now we proceed to the proof of Theorem \ref{fluc-2d}. 

We first remark that by Woess \cite[Theorems 14.12 and 14.19]{W},   
\begin{Prop}\label{Nash-2nd}
Assume $G$ satisfies that $|B(x,n)| \le C n^2$ and $p_n (x,x) \le C n^{-1}$. 
Assume $G^{\prime}$ is roughly isometric to $G$. 
Then, 
$G^{\prime}$ also satisfies that $|B(x,n)| \le C n^2$ and $p_n (x,x) \le C n^{-1}$.  
\end{Prop}

\begin{proof}[Proof of Theorem \ref{fluc-2d}]
First we show \eqref{fluc-2d-p0}.  
Let $| \cdot |_{\infty}$ be the infinite norm of $\Z^2$. 
Let $\widetilde \Z^2 = (\Z^2, E(\widetilde \Z^2))$ be the graph whose vertices and edges are $\Z^2$ and $\left\{\{x,y\} :  |x-y|_{\infty} = 1\right\}$.  
Then, we can show that 
\[ \lim_{n \to \infty} n P_{\widetilde \Z^2}^{0}(S_{2n} = 0) = \frac{2}{3 \pi},  \] 
and hence, 
\[ \lim_{n \to \infty} \frac{\log n}{n} E_{\widetilde \Z^2}^0 [R_n] = \frac{3}{2} \pi.  \]

We can show Theorem \ref{fm-2d} for $\widetilde \Z^2$ instead of $\Z^2$. 
Hence, if $G$ is a finite modification of $\widetilde \Z^2$, not $\Z^2$, then, for every $o \in G$, 
 \[ \lim_{n \to \infty} \frac{\log n}{n} E_{G}^{o} [R_n] = \frac{3}{2} \pi.  \]

Then, 
the desired graph $G$ can be constructed in a manner similar to the proof of \cite[Theorem 1.3]{O14}, so we omit details.       
By the construction, we have that 
\[ \liminf_{n \to \infty} \frac{E_{G}^0 [R_n]}{n / \log n}  = r_{\inf} = \pi <  \frac{3}{2} \pi = r_{\sup} = \limsup_{n \to \infty} \frac{E_{G}^0 [R_n]}{n / \log n}.  \]
Thus we have \eqref{fluc-2d-p0}. 
The graph $G$ is roughly isometric to $\Z^2$, and by using this and Proposition \ref{Nash-2nd}, we have  that $G$ is uniformly weakly recurrent. 

Now we show \eqref{SLLN-fluc}. 

\begin{Lem}\label{SLLN-fluc-lem}
Assume that for some $d > 1$, $G$ is Ahlfors $d$-regular and satisfies the sub-Gaussian heat kernel estimates with walk dimension $d$. 
\begin{equation}\label{different-mean} 
\liminf_{n \to \infty} \frac{E^x[R_n]}{n / \log n} < \limsup_{n \to \infty} \frac{E^x[R_n]}{n / \log n}. 
\end{equation} 
Then, 
\begin{equation}\label{different-c}
c_{\inf} < c_{\sup},  
\end{equation}
where $c_{\inf}$ and $c_{\sup}$ are constants introduced in Theorem \ref{SLLN}.  
\end{Lem}

\begin{proof}[Proof of Lemma \ref{SLLN-fluc-lem}]
Assume that \eqref{different-c} fails. 
Then, we can let $c = c_{\inf} = c_{\sup}$ and 
\[ \lim_{n \to \infty} \frac{R_n}{n / \log n} = c, \ \textup{ $P^x$-a.s.} \]
By the estimate \eqref{decay-fast-1} in the final part of the proof of Theorem \ref{WLLN-UC} (i), we have that 
\[ \sup_{x \in V(G), n \ge 1} \left(\frac{\log n}{n}\right)^2 E^x [R_n^2] < +\infty, \]
and hence $\{R_n/(n/\log n)\}_n$ is uniformly integrable. 
Hence,
\[ \lim_{n \to \infty} \frac{R_n}{n / \log n} = c, \ \textup{ in } L^1(P^x). \]
This contradicts \eqref{different-mean}.   
\end{proof}

Since by \cite{BB04} sub-Gaussian heat kernel estimates are stable with respect to rough isometries, 
 $G$ is Ahlfors $d$-regular and satisfies the sub-Gaussian heat kernel estimates with walk dimension $d$ for some $d > 1$, and we have \eqref{SLLN-fluc}. 
\end{proof}

\begin{Rem}
(i) In the proof of Theorem \ref{fm-2d}, we adopt the strategy similar to the one by showing the corresponding result for the Wiener sausage in a bounded modification for $\R^2$ (\cite[Theorem 1.4 (ii)]{O-ws}).
However, for the Wiener sausage case, it is not easy to establish a fluctuation result  corresponding Theorem \ref{fluc-2d} even if we have the convergence result for bounded modifications like \cite[Theorem 1.4 (ii)]{O-ws}.  
The technical difficulty for the continuous framework arises from the fact that the diffusion process which we consider can be arbitrarily far from its starting point in arbitrarily small time.  
This difficulty also appears if we consider continuous-time random walks. 
This is contrary to our case where it always holds that $R_n \le n$. \\
(ii) We are not sure whether a sequence of a scaled on-diagonal heat kernel $\left(\sum_{k=0}^{n} p_{k}(o,o) / \log n\right)_n$ fluctuates or not for the graph $G$ in the above proof of Theorem \ref{fluc-2d}.   
\end{Rem}

\section{Examples}

In this section, we consider examples of uniformly weakly recurrent graphs.  

\subsection{Isoradial graphs}

In this subsection, we show that Theorems \ref{WLLN-UC}, \ref{Lower} and \ref{SLLN} hold for graphs which are roughly isometric to isoradial graphs with the bounded-angles property. 

\begin{Def}[isoradial graph; Kenyon \cite{K02}]
We say that a planar graph $G$ embedded in $\mathbb C$ is {\it isoradial} if each face is inscribed in a circle with radius one. 
\end{Def}

This is equivalent to rhombic lattices, which were introduced by Duffin \cite{D68}.  
It was shown by Kenyon-Schlenker \cite{KS04} that many planar graphs admit isoradial embeddings. 
Note that an isoradial graph is necessarily infinite. 

\begin{Def}[the bounded-angles property]
Let $\{x,y\}$ be an edge of a connected isoradial graph. 
Then there exist exactly two circles which contain $\{x,y\}$.
$\{x,y\}$ and the centers of the circles form a rhombus.  
Let $2\theta_x$ be the angles at $x$  of the rhombus. 
We say that the bounded-angles property holds if there exists a constant $c \in (0, \pi/4)$ such that for every $x \in G$, 
\[ c \le \theta_x \le \frac{\pi}{4} - c. \]
\end{Def}

\begin{Exa}
The 2-dimentional square, triangle, and hexagonal lattices are all isoradial graphs satisfying the bounded-angles property. 
\end{Exa}

\begin{Lem}\label{AP}
If an isoradial graph satisfy the bounded-angles property, then, 
there exist two positive constants $c_1, c_2$ such that for every edge $\{x, y\} \in E(G)$, 
$$c_1 \le |x-y| \le c_2,$$
and 
$$c_1  \le d_{\star}(x,y) \le c_2,$$
where $d_{\star}(x,y)$ denotes the length of the dual edge of $\{x,y\}$. 
\end{Lem} 

\begin{Def}[boundary]
For $\Omega \subset V(G)$, we let 
$$\partial \Omega := \left\{ \{x,y\} \in E(G) | x \in \Omega, y \in V(G) \setminus \Omega \right\}.$$
\end{Def}

\begin{Prop}
An isoperimetric inequality of order 2 holds for every isoradial graph $G$ with  the bounded-angles property, specifically,  
there exists a positive constant $c$ such that for every non-empty finite subset $\Omega$ of $V(G)$, 
\[ \left|\partial \Omega\right| \ge c |\Omega|^{1/2}. \]
\end{Prop}

\begin{proof}
This assertion follows from Lemma \ref{AP}, the duality, and the isoperimetric inequality on the Euclid plane. 

Let $K$ be the closure of the domain surrounded by $\partial \Omega$. 
Then by the isoperimetric inequality on the Euclid plane, we have that
\[ L(\partial K) \ge 2 \sqrt{\pi} |K|^{1/2}, \]
where $L(\partial K)$ denotes the arc length of $\partial K$ and $|K|$ denotes the Lebesgue measure of $K$. 

Since $\partial K$ consists of the dual edges of $\partial \Omega$, 
it follows from Lemma \ref{AP} that 
\[ |\partial \Omega| = |\{\textup{dual edges of $\partial \Omega$}\}| \ge \frac{1}{c_2} L(\partial K), \]
where $c_2$ is the constant in Lemma \ref{AP}. 

It follows from Lemma \ref{AP} that 
\[  |K| \ge \frac{\pi}{4} c_1^2 |\Omega|, \]
where $c_1$ is the constant in Lemma \ref{AP}. 

Thus we have that 
\[  |\partial \Omega| \ge \frac{c_1}{c_2} \pi  |\Omega|. \]
\end{proof}

The above assertion implies the Nash inequality of order one. 

\begin{Prop}[Nash inequality of order 1]
For every isoradial graph $G$ with  the bounded-angles property, 
there exists a positive constant $C$ such that for every $n \ge 1$, 
\[ p_{n}(x,x) +p_{n+1}(x,x)  \le Cn^{-1}. \]
\end{Prop}

Now we consider the relative isoperimetric inequality for isoradial graphs with the bounded-angles property.  

\begin{Lem}[relative isoperimetric inequality]
The  relative isoperimetric inequality  holds, that is, 
for every finite subset $B$ of $V(G)$, 
there exists $C_B$ such that for every finite $A \subset B$ such that $0 < \mu(A) \le \mu(B)/2$, 
\[ |\partial_B A| \ge C_B \mu(A), \]
where we let 
\[ |\partial_B A| := \sum_{x \in A, y \in B\setminus A} \mu_{xy}. \]
\end{Lem}

\begin{proof}
Let $K_A$ and $K_B$ be closed sets surrounded by the dual edges of $\partial A$ and $\partial B$. 
For closed subsets $C \subset D$ in the Euclid plane, let $P(C, D)$ be the perimeter of $C$ in $D$.  
By the relative isoperimetric inequality in the Euclid plane, 
\[  P(K_A, K_B)^2 \ge C_{K_B} \min\left\{|K_A|, |K_B \setminus K_A| \right\} \ge C |K_A|. \]
Here $|K_A| \ge |K_B \setminus K_A| $ might happen, however by $\mu(A) \le \mu(B)/2$ and Lemma \ref{AP}, 
there exists a positive universal constant $c$ such that 
\[ |K_B \setminus K_A| \ge c|K_A|. \]

By this and Lemma \ref{AP}, 
\[  |\partial_B A| \ge c C_{K_B} P(K_A,  K_B) \ge  C |K_A| \ge C |A|. \]
\end{proof}

Since the relative isoperimetric inequality implies the Poincar\'e inequality (see \cite[Section 3.3]{Ku} for this fact), 
we have that 
\begin{Prop}[Gaussian heat kernel estimates]
There exist four constants $c_3, c_4, c_5, c_6 > 0$ such that for every $n \ge 1$ and every $x, y \in V(G)$, 
\[ p_{n}(x,y) \le \frac{c_3}{|B(x,n^{1/2})|} \exp\left(- c_4 \frac{d(x,y)^{2}}{n}\right), \]
and, 
\[ p_{n}(x,y) + p_{n+1}(x,y) \ge \frac{c_5}{|B(x,n^{1/2})|} \exp\left(- c_6  \frac{d(x,y)^{2}}{n} \right). \]
\end{Prop}

Since the Gaussian heat kernel estimates and the Ahlfors regularity are both stable with respect to rough isometries (see \cite[Theorem 3.1]{GT} and \cite{BB04} for this fact), 
Theorems \ref{WLLN-UC} and \ref{SLLN} hold also for graphs which is roughly isometric to isoradial graphs with the bounded-angles property.

\subsection{Graphs with spectral dimension two and polynomial volume growth with arbitrarily degree}

We say that $G$ satisfies $(V_d)$, or is {\it Ahlfors $d$-regular}, if 
there exists a constant $c > 1$ such that for every $x \in G$ and $r \ge 1$,   
\[ c^{-1}r^{d} \le |B(x,r)| \le c r^d. \]
We show that for every $d \ge 2$, there exists an Ahlfors $d$-regular graph satisfying \eqref{posi-fi}.  

We say that $G$ satisfies $(E_d)$ if      
there exists a constant $c > 1$ such that for every $x \in G$ and $r \ge 1$,   
\[ c^{-1}r^{d} \le E^x \left[T_{B(x,r)^c}\right] \le c r^d. \]

We say that $G$ satisfies the {\it elliptic Harnack inequality} (EHI) if there exists a constant $c > 1$ such that for every $x \in G$ and $r \ge 1$ and every non-negative function $h : B(x, 2r+1) \to [0, +\infty)$ which is harmonic on $B(x, 2r)$,  
\[ \sup_{y \in B(x,r)} h(y) \le c \inf_{y \in B(x,r)} h(y). \]

Let $d \ge 2$. Then, by \cite[Theorem 2 and Lemma 1.3]{B04-1}, 
there exists a graph $G_d$ satisfying $(V_d), (E_d)$ and (EHI), and, $G_d$ satisfies sub-Gaussian heat kernel estimates with spectral dimension two. 
By using \cite[Theorem 3.1]{GT} and \cite{BB04}, 
we see that every graph which is roughly isometric to $G_d$ satisfies sub-Gaussian heat kernel estimates with spectral dimension two, and hence, 
for some $d > 1$, $G$ is Ahlfors $d$-regular and satisfies the sub-Gaussian heat kernel estimates with walk dimension $d$. 

Recently, Murugan \cite[Examples 7.3 and 7.5]{M} give examples of infinite planer graphs satisfying sub-Gaussian heat kernel estimates with spectral dimension two.   
\cite[Examples 7.3 and 7.5]{M} are simpler than the examples of \cite{B04-1} stated above. 
\cite[Example 7.3]{M} deals with a graphical snowball which is Ahlfors $\log_3 (13)$-regular and whose walk dimension is $\log_3 (13)$.  
\cite[Example 7.5]{M} deals with a graph version of a regular pentagonal tiling which is Ahlfors $\log_2 (6)$-regular and whose walk dimension is $\log_2 (6)$.

\begin{Rem}[Weakly recurrent, but not uniformly weakly recurrent graph] 
The unique infinite cluster of 2-dimensional supercritical percolation is a.s. weakly recurrent and its spectral dimension is two. 
However, almost surely, it is not uniformly weakly recurrent, because it satisfies Gaussian heat kernel estimates, and has one-dimensional object having arbitrarily size. 
See \cite{B04-2} for more details.      
\end{Rem} 

\subsection{Application to law of iterated logarithms for random walks on Lamplighter graphs} 

\cite{KuN} considered the laws of the iterated logarithms (LIL) for a switch-walk-switch random walk on a lamplighter graph under the condition that the random walk on the underlying graph satisfies sub-Gaussian heat kernel estimates. 
\cite[Theorem 2.4]{KuN} states LILs for the random walk on a lamplighter graph in the case that the spectral dimension of the underlying graph is {\it not} two. 
They found some relationship between the random walk on a lamplighter graph and the range of the random walk on the underlying graph. 
Their proof uses an improved version of a result for the range of random walk \cite[Corollary 2.3]{O14}. 
See \cite[Proposition 5.7]{KuN}.
However, the case that the spectral dimension is two, which we deal with below, remains to be open.  
See \cite[Remark 2.5 (4)]{KuN}. 

We refer readers to \cite{KuN} for terminologies. 
Denote  the wreath product of a graph $G$ and $\{0,1\}$ by $\{0,1\} \wr G$. 
Let $d$ be the graph distance on $\{0,1\} \wr G$, and $\{Y_n\}_n$ be the simple random walk on $\{0,1\} \wr G$.  

\begin{Prop}
Let $G$ be a uniformly weakly recurrent graph. 
Then, for every vertex $x$ of $\{0,1\} \wr G$ and every $\epsilon > 0$, \\
(i) 
\begin{equation}\label{lamp-u} 
\limsup_{n \to \infty} \frac{d(Y_0, Y_n)}{n / \log n} \le \left(2\max_{x \in V(G)} \deg(x) + 1\right) r_{\sup}, \ \textup{$P_{\{0,1\} \wr G}^x$ -a.s.} 
\end{equation} 
and, 
\begin{equation}\label{lamp-l}   
\lim_{n \to \infty} P_{\{0,1\} \wr G}^x \left( \frac{d(Y_0, Y_n)}{n / \log n} \le \frac{r_{\inf}}{4}  - \epsilon \right) = 0, 
\end{equation} 
(ii) If furthermore $G$ is Ahlfors $d$-regular for some positive $d > 0$, then, 
\begin{equation}\label{lamp-l-as}   
\liminf_{n \to \infty} \frac{d(Y_0, Y_n)}{n / \log n} \ge \frac{r_{\inf}}{4}, \ \textup{$P_{\{0,1\} \wr G}^x$ -a.s.} 
\end{equation} 
\end{Prop} 

This assertion follows from Theorems \ref{WLLN-UC} and \ref{Lower}, and \cite[Propositions 5.1 and 5.2]{KuN}. 

\begin{Rem}
(i) For \eqref{lamp-u}, \eqref{lamp-l} and \eqref{lamp-l-as}, 
$G$ does not need to satisfy both of \eqref{sGHK-upper} and \eqref{sGHK-lower} as in \cite[Assumption 2.2 (4)]{KuN}.  
See the proofs of \cite[Propositions 5.1 and 5.2]{KuN}. \\
(ii) \eqref{lamp-u}, \eqref{lamp-l} and \eqref{lamp-l-as} hold for the graphs in the above subsections.    
\end{Rem} 

\appendix 

\section{Intersection of two independent random walks on graph}

In this section we consider intersections of two independent random walks starting at a common point. 
This case is relatively easier to handle than Section 3, but on the other hand the arguments in this section are similar to Section 3. 

\begin{Prop}\label{intersect-main}
Let $S^{(1)}$ and $S^{(2)}$ be two independent simple random walks on $G$.  
Let $P^{x,y}$ be the joint law of $S^{(1)}$ and $S^{(2)}$ whose starting points are $x$ and $y$ respectively.  
Then there exists a positive constant $C$ such that for every $\epsilon > 0$ and every $x \in V(G)$,
\[ P^{x,x}\left( \left| \{S^{(1)}_0, \dots, S^{(1)}_n\} \cap  \{S^{(2)}_0, \dots, S^{(2)}_n\}  \right| \ge \frac{\epsilon n}{\log n}\right) \le \frac{C (\log\log n)^2}{\epsilon^2 (\log n)^{2}} \]
\end{Prop}

As in \cite{L}, the estimate in Proposition \ref{intersect-main} is attributed to obtaining an upper bound for the expectation of the intersection. 

\begin{Prop}\label{intersect-expectation}
There exists a positive constant $C$ such that for every $x \in V(G)$ and $n \ge 2$, 
\[  E^{x,x}\left[ \left| \{S^{(1)}_0, \dots, S^{(1)}_n\} \cap  \{S^{(2)}_0, \dots, S^{(2)}_n\}  \right| \right] \le \frac{C n \log\log n}{(\log n)^2}\]
\end{Prop}

For the case that $G = \Z^2$, we have a better estimate (see \cite[Theorem 5.1]{L}).   
We conjecture that $\log\log n$ in the numerator of the right hand side of the above inequality could be replaced with $1$. 
We use estimates for hitting time distributions on graphs (cf. \cite{B17}).

\begin{proof}
By  \eqref{Ahlfors} and Lemma \ref{hitting-self},   
we have that 
 \[ E^{x,x}\left[ \left| \{S^{(1)}_0, \dots, S^{(1)}_n\} \cap  \{S^{(2)}_0, \dots, S^{(2)}_n\}  \right| \right] = \sum_{y} P^x (T_y \le n)^2 \]
\[ \le \left|B\left(x, \left(\frac{n}{(\log n)^2}\right)^{1/d}\right) \right| + \sup_{y \notin B(x, (\frac{n}{(\log n)^2})^{1/d})} P^x (T_y \le n) \sum_{y} P^x (T_y \le n)\]
\[ = O\left(\frac{n}{(\log n)^2}\right) + O\left(\frac{\log\log n}{\log n}\right) E^x [R_n].  \] 
By this and \eqref{exp-upper}, 
we have the assertion. 
\end{proof}

The following lemma is related with the moment method. 

\begin{Lem}\label{moment} 
For every $p \in \mathbb{N}$, 
\[ \sup_{x \in V(G)} E^{x,x} \left[ \left|\{S^{(1)}_0, \dots, S^{(1)}_n\} \cap \{S^{(2)}_0, \dots, S^{(2)}_n\} \right|^p \right] \]
\[\le (p!)^2 \left( \sup_{x \in V(G)} E^{x,x} \left[ \left|\{S^{(1)}_0, \dots, S^{(1)}_n\} \cap \{S^{(2)}_0, \dots, S^{(2)}_n\} \right| \right] \right)^p. \]
\end{Lem}

Theorem \ref{intersect-main} easily follows from Proposition \ref{intersect-expectation} and Lemma \ref{moment}. 
The proof of Lemma \ref{moment} is same as in the proof of \cite[Lemma 3.1]{LR}, so we omit the proof.\\

\noindent{\it Acknowledgements.} The author wishes to give his thanks to Takashi Kumagai for notifying him of the application to the lamplighter random walk as in \cite{KuN} and comments on this topic.   
This work was supported by JSPS KAKENHI 16J04213, 18H05830, 19K14549 and by the Research Institute for Mathematical Sciences, a Joint Usage/Research Center located in Kyoto University.

\end{document}